\documentclass{article}

\usepackage{arxiv}

\usepackage[utf8]{inputenc} 
\usepackage[T1]{fontenc}    
\usepackage{booktabs}       
\usepackage{amsmath}
\usepackage{amsfonts}       
\usepackage{amssymb}
\usepackage{amsthm}
\usepackage{nicefrac}       
\usepackage{microtype}      
\usepackage{graphicx}
\usepackage{doi}
\usepackage{natbib}
\usepackage[english]{babel}
\usepackage{enumitem}
\usepackage{graphicx}
\usepackage{wasysym}
\usepackage{multimedia}
\usepackage{subfig}
\usepackage{amsfonts}
\usepackage{libertine}
\usepackage{lmodern}
\usepackage{fancyhdr}
\usepackage{color}
\usepackage{dsfont}
\usepackage{bbm}
\usepackage{multicol}
\usepackage{bm} 
\usepackage{ulem}
\usepackage{tabularx} 
\usepackage{mathrsfs}
\usepackage{stmaryrd}
\usepackage{array,multirow,makecell}
\setcellgapes{1pt}
\makegapedcells%
\usepackage{systeme}
\usepackage{comment}
\usepackage{tcolorbox}
\usepackage{cases}
\usepackage{algorithm}
\usepackage{algorithmic}
\usepackage{mathtools}
\usepackage{colortbl}

\usepackage[toc,page,title,titletoc,header]{appendix} 

\newtheorem{theorem}{Theorem}
\newtheorem{lemma}{Lemma}
\newtheorem{definition}{Definition}

\newtheorem{remark}{Remark}

\usepackage{afterpage}

\newcommand{\db}{qp} 
\newcommand{\dg}{r} 

\title{Posterior contraction rates in a sparse non-linear mixed-effects model}

\date{} 					

\author{ Marion Naveau \\ Université Paris-Saclay, AgroParisTech, INRAE, UMR MIA-Paris-Saclay, 75005, Paris, France. \\
         Université Paris-Saclay, INRAE, MaIAGE, 78350, Jouy-en-Josas, France. \\
        \And
	Maud Delattre \\ Université Paris-Saclay, INRAE, MaIAGE, 78350, Jouy-en-Josas, France. \\
	\And 
	Laure Sansonnet \\ Université Paris-Saclay, AgroParisTech, INRAE, UMR MIA-Paris-Saclay, 75005, Paris, France. \\ 
    Sorbonne Université, Université Paris Cité, CNRS, LPSM, 75005 Paris, France. \\
}

\theoremstyle{plain}

\theoremstyle{definition}
\newtheorem*{example}{Example}

\newtheorem{assumption}{Assumption}

\renewcommand{\headeright}{} 
\renewcommand{\undertitle}{} 
\renewcommand{\shorttitle}{Posterior contraction rates in a sparse non-linear mixed-effects model}

\hypersetup{
pdftitle={Posterior contraction rates in a sparse non-linear mixed-effects model},
pdfsubject={q-bio.NC, q-bio.QM},
pdfauthor={Marion ~Naveau, Laure ~Sansonnet, Maud ~Delattre},
pdfkeywords={Posterior contraction,  Spike-and-slab prior,  Non-linear mixed-effects models,  Bayesian variable selection,  High-dimensional data},
}

\def\keywordname{{\bfseries Keywords}}%
\def\keywords#1{\par\addvspace\medskipamount{\rightskip=0pt plus1cm
\def\and{\ifhmode\unskip\nobreak\fi\ $\cdot$
}\noindent\keywordname\enspace\ignorespaces#1\par}}

\begin{document}
\maketitle

\begin{abstract}
Recent works have shown an interest in investigating the frequentist asymptotic properties of Bayesian procedures for high-dimensional linear models under sparsity constraints. However, there exists a gap in the literature regarding analogous theoretical findings for non-linear models within the high-dimensional setting. The current study provides a novel contribution, focusing specifically on a non-linear mixed-effects model. In this model, the residual variance is assumed to be known, while the regression vector and the covariance matrix of the random effects are unknown and must be estimated. The prior distribution for the sparse regression coefficients consists of a mixture of a point mass at zero and a Laplace distribution, while an Inverse-Wishart prior is employed for the covariance parameter of the random effects. First, the effective dimension of this model is bounded with high posterior probabilities. Subsequently, we derive posterior contraction rates for both the covariance parameter and the prediction term of the response vector. Finally, under additional assumptions, the posterior distribution is shown to contract for recovery of the unknown sparse regression vector at a rate similar to that established in the linear case.
\end{abstract}


\section{Introduction}

Recent statistical literature has shown a keen interest in estimating high-dimen\-sional models under sparsity assumptions, with different approaches proposed over the past few decades in both Bayesian and frequentist frameworks. The developed methodologies are numerous and use a large variety of techniques such as convex and non-convex penalization techniques, shrinkage methods and sparsity-inducing priors. In Bayesian analysis, a category of proposed priors includes those defined as mixtures of two distributions, commonly referred to as spike-and-slab priors. These priors have proven to be useful and relevant in many high-dimensional applications as demonstrated in \citet{george, george97, tadesse2021handbook}.

The frequentist asymptotic properties of Bayesian sparse linear regression models with various types of mixture priors have been widely investigated, particularly in \citet{narisetty_bayesian_2014}, \citet{castillo_bayesian_2015}, and \citet{rockova_SSL} with the spike-and-slab Gaussian prior, the discrete spike-and-slab prior, and the spike-and-slab lasso prior respectively. Subsequently, these investigations were extended to multivariate linear regression with an unknown residual covariance matrix, as discussed by \citet{ning_bayesian_2020} with a discrete spike-and-slab prior and \citet{shen_posterior_2022} with a multivariate spike-and-slab lasso prior. The classical techniques for determining posterior contraction rates (see \textit{e.g.} \citet{castillo_bayesian_2015}) face limitations when the residual covariance matrix is unknown. In such scenarios, the general theory \citep{ghosal2017fundamentals} based on the average squared Hellinger distance proves inadequate for obtaining rates in terms of the Euclidean norm for parameters. To overcome this difficulty, an alternative approach has been introduced, leveraging the average Rényi divergence of order 1/2. As underscored by \citet{ning_bayesian_2020}, this method enables the construction of exponentially powerful tests that are required by the general theory \citep{ghosal2017fundamentals}, facilitating a more effective determination of posterior contraction rates in Bayesian analysis. Another theoretical aspect requires adaptation to the general theory when the residual covariance matrix is unknown. Indeed, classical proofs require lower bounds for prior mass around true parameter values, but when the residual covariance matrix is unknown, this condition can only be fulfilled if the true parameter set is bounded, as discussed in works of \citet{ning_bayesian_2020} and \citet{jeong_posterior_2021, jeong_unified_2021}.

Recent advancements have expanded the study of estimation and selection properties to more complex models than sparse linear regression models, such as sparse generalized linear models \citep{jiang_bayesian_2007, jeong_posterior_2021} or sparse linear regression models with nuisance parameters \citep{jeong_unified_2021}. To our knowledge, there are no similar theoretical results for non-linear models in high-dimensional contexts. The absence of theoretical results in this domain may reflect the inherent challenges and complexities associated with extending such analyses to non-linear models. The present paper fulfills this gap and provides a contribution in this direction, focusing on a specific non-linear model which also contains random effects. Mixed-effects models have been introduced to analyze observations collected repeatedly on several individuals in a population of interest, commonly encountered in fields such as pharmacokinetics or biological growth modeling for example \citep{pinheiro, lavielle}. These models, which are generally non-linear, may use high-dimensional covariates to describe inter-individual variability. Our paper deals with a generalization of the linear mixed-effects model to a non-linear marginal version where the fixed effects are non-linear functions of the regression parameter, while the random effects are incorporated into the model in a linear manner (see \textit{e.g.} \citet{demidenko2013mixed}). Such non-linear marginal mixed-effects models are easier to handle than more general non-linear mixed-effects models because the mean and the covariance matrix of the response variable are explicit. Recently, some computational solutions have been developed for estimating and selecting variables in high-dimensional context for this type of model \citep{ollier, naveau_bayesian_2023}. However, despite their practical appeal, there has been a lack of theoretical exploration concerning non-linear marginal mixed models in high-dimensional context. In this paper, posterior contraction rates are obtained for both the covariance matrix and the prediction term in a high-dimensional setting by using a mixture of a point mass at zero and a Laplace distribution prior on the regression coefficients, and an inverse Wishart prior on the covariance matrix. Furthermore, we extend these results to the regression coefficients themselves, under additional assumptions that ensure the identifiability of these coefficients. To obtain these results from the general theory, new arguments had to be developed to overcome the difficulty of the model's non-linearity.

This paper is organized as follows. Section~\ref{sec_model} describes the non-linear marginal mixed model to introduce the notation, defines the prior distributions, along with the necessary assumptions. Section~\ref{sec_results} provides the main results on the posterior contraction with an example of non-linear marginal mixed-effects model that satisfies our conditions. We also provide a conclusion in Section~\ref{sec_conclusion}. Finally, the proofs of the theorems are given in Section~\ref{sec_proofs}. Proofs of technical lemmas are postponed in the appendices. 

\paragraph{Notation}

This paragraph describes the notations used in this paper for a generic matrix $A=\left(a_{ij}\right)_{i,j}$ and a generic vector $\theta \in \mathbb{R}^k$. We note $S_{\theta}=\{j | \theta_j \neq 0\}$ the support of $\theta$ and $s_{\theta}=|S_{\theta}|$ its cardinal. The Euclidean norm, the $\ell_1$-norm and the infinity norm are respectively noted $\| \theta \| _2=\left(\sum_{i=1}^k \theta_i^2\right)^{1/2}$, $\|\theta\|_1=\sum_{i=1}^k |\theta_i|$, and $\| \theta\| _{\infty}=\max_i | \theta_i| $. The transpose of $A$ is denoted by $A^\top$. For a square matrix $A$, we note $\rho_{min}(A)$ and $\rho_{max}(A)$ the minimum and maximum eigenvalues of $A$, respectively. The spectral norm of a matrix $A$ is denoted $\| A \| _{sp} = \rho^{1/2}_{max}(A^\top A)$, and the Frobenius norm is noted $\| A \| _{F} = \text{Tr}(A^\top A)^{1/2} =$ $(\sum_{i,j} a_{ij}^2)^{1/2}$. The matrix norm $\|A\|_{*}$ is defined as $\|A\|_{*}= \max_j \| A_{\cdot j}\|_2$ for $A_{\cdot j}$ the $j$-th column of $A$. The identity matrix of size $m$ is denoted $I_m$. The set of real symmetric positive-definite matrices is denoted by $\mathcal{S}_n^{++}$.

For sequences $a_n$ and $b_n$, the notation $a_n \lesssim b_n$ means that for $n$ large enough $a_n$ is bounded above by a constant multiple of $b_n$, \textit{i.e.} $a_n \leq C b_n$ for $n$ large enough, where $C>0$ is independent of $n$, and $a_n \asymp b_n$ means $a_n \lesssim b_n \lesssim a_n$. We denote $a_n=o(b_n)$ if $a_n/b_n$ tends to 0 when $n \to \infty$.

\section{Model description}
\label{sec_model}

\subsection{Non-linear marginal mixed-effects model}
\label{ss_sec_model}
Mixed-effects models are sophisticated multivariate statistical models employed to analyze repeated observations, usually collected over time, on $n$ statistical subjects, incorporating both fixed and random effects into the model for accurate description \citep{demidenko2013mixed, lavielle}. We consider the following mixed-effects model which states that for all $i \in \{1,\ldots,n\},j \in \{1,\ldots,m_i\}$:

\begin{equation}
    \label{NLMMij}
    Y_{ij}= f(\varphi_i,t_{ij}) + Z_i(t_{ij})^\top \xi_i + \varepsilon_{ij}.  
\end{equation}
In the above equation, $\xi_i \in \mathbb{R}^r$ is a vector composed of $r$ random effects, $\varepsilon_{ij} \in \mathbb{R}$ is an error term, $t_{ij}$ is the time at which the $j$-th observation of individual $i$ - $Y_{ij}$ - is recorded, $m_i$ is the number of observations for subject $i$, and $Z_i(t_{ij}) \in \mathbb{R}^r$ is a set of potentially time-dependent explanatory variables that relate the observations to the random effects. Moreover, $f : \mathbb{R}^q \times \mathbb{R} \rightarrow \mathbb{R}$ is a known and potentially non-linear function in terms of the $q$ components of its parameter $\varphi$ which itself is supposed to differ from one individual $i$ to the other in $\{1,\ldots,n\}$ according to a linear combination of covariates organized in a vector $V_i \in \mathbb{R}^{p}$ with coefficients given by $\tilde{\beta} \in \mathbb{R}^{q \times p}$: $\varphi_i = \tilde{\beta} V_i$. For a more cohesive presentation, denote $Y_i = (Y_{i1},\ldots,Y_{im_i})^\top  \in \mathbb{R}^{m_i}$ the vector of observations for subject $i$, $\epsilon_i = (\epsilon_{i1},\ldots,\epsilon_{im_i})^\top  \in \mathbb{R}^{m_i}$ the vector of residuals for subject $i$, the stacked matrix $Z_i=(Z_i(t_{i1})^\top,\ldots,Z_i(t_{im_i})^\top)^\top \in \mathbb{R}^{m_i \times r}$, the $\mathbb{R}^{q \times (q p)}$ block matrix $X_i = \begin{pmatrix}
 V_i^\top & 0 & \ddots & 0\\
 0 &  V_i^\top & \ddots & 0\\
  0 & 0 & \ddots &  V_i^\top 
\end{pmatrix}$ and $\beta=\mathrm{Vec}(\tilde{\beta}^\top) \in \mathbb{R}^{\db}$ the vectorization of $\tilde{\beta}^\top$. Note that vectorizing the coefficient matrix simplifies the mathematical process required to obtain the results presented in Section~\ref{sec_results}. We have $\varphi_i= \tilde{\beta} V_i = X_i \beta$. Let us also define the vector functions $f_i : \mathbb{R}^q \rightarrow \mathbb{R}^{m_i}$, $i \in \{1,\ldots,n\}$, such that for every $\varphi \in \mathbb{R}^q$:
$$
f_i(\varphi) = (f(\varphi,t_{i1}), f(\varphi,t_{i2}), \ldots, f(\varphi,t_{im_i}))^{\top}.
$$
Then the model defined by~\eqref{NLMMij} can be written as
\begin{equation}
    \label{NLMM}
    Y_{i}= f_i(X_i \beta) + Z_i \xi_i + \varepsilon_{i}.
\end{equation}
In addition, we assume that $\varepsilon_i$ are independent random vectors with distribution $\mathcal{N}_{m_i}(0,\sigma^2 I_{m_i})$ and $\xi_i$ are \textit{i.i.d.} random vectors with distribution $\mathcal{N}_r(0,\Gamma)$.
The term $X_i \beta$ is crucial in these models, as the sparsity of $\beta$ provides insight into the covariates associated with the observed variability between individuals. It is important to emphasize that \eqref{NLMM} defines a broad class of models, encompassing both linear mixed-effects models and a wide variety of nonlinear mixed-effects models. While opting for $f_i(X_i\beta)=A(t_{i1},\ldots,t_{im_i})X_i\beta$, with $A(t_{i1},\ldots,t_{im_i}) \in \mathbb{R}^{m_i \times q}$ a known matrix potentially depending on time, yields the standard linear mixed-effects model, it is common in many applications to select a non-linear function $f$. The choice of the function $f$ can be guided by various criteria, such as the shape of the observed dynamics or a priori knowledge about the phenomenon under study. In the latter case, a mechanistic model is used to reflect known underlying mechanisms. This approach is common in pharmacology for example, where models are typically mechanistic by nature.

\begin{example}[logistic growth model]
A commonly used nonlinear function $f$ in various applied fields of mixed-effects models is the logistic function that corresponds, in the model defined by~\eqref{NLMMij}, to
$$f(\varphi_i,t_{ij})=\dfrac{\varphi_{i1}}{1+\exp\left(\varphi_{i2}(t_{ij}-\varphi_{i3})\right)},$$ where $\varphi_i=(\varphi_{i1},\varphi_{i2},\varphi_{i3})^\top$ are individual-specific parameters representing the upper asymptote, growth rate, and inflection point. This type of sigmoidal growth curve is widely used in plant and animal breeding, pharmacokinetics, and ecological modeling to capture individual-level dynamics over time within a mixed-effects framework \citep{pinheiro, lavielle}.
\end{example}

When $f$ is non-linear, model defined by~\eqref{NLMM} is alternatively referred to as the non-linear marginal mixed-effects model (as discussed in \citet{demidenko2013mixed}). The term "marginal mixed-effects model" is derived from the fact that, unlike numerous other non-linear mixed-effects models, both the expectation and variance of the observations possess an explicit expression. The distribution of $Y_i$ defined through \eqref{NLMM} is thus fully characterized: 
\begin{equation}
\label{NLMM_compact}
    Y_i \sim \mathcal{N}(f_i(X_i \beta), \Delta_{\Gamma,i}) \text{, where }\Delta_{\Gamma,i}=Z_i \Gamma Z_i^\top + \sigma^2 I_{m_i}.
\end{equation}
It is worth noting that mixed-effects machine learning methods as in \cite{krennmair2022flexible,kilian2023mixed} rely on this formulation of a nonlinear mixed-effects model, with the key distinction that the function $f$ is nonparametric, unlike in our study framework.
As briefly mentioned above, when $f$ is parametric, a significant focus lies on the covariate selection process in $X_i$, as it allows establishing connections between inter-individual variability and measured individual characteristics. The motivation for our work stems from this question, and its specific aim is to estimate $\beta \in \mathcal{B}:=\mathbb{R}^{\db}$ and $\Gamma \in \mathcal{H}:=\mathcal{S}_n^{++}$ in an high-dimensional setting where $p >> n$ and to obtain posterior contraction results. To this end, we assume that the residual variance $\sigma^2$, the number $q$ of individual parameters $\varphi_i$, and the number $r$ of true random effects are fixed and known. We establish below appropriate priors to achieve these goals.

\subsection{Prior specification}
\label{sec_prior}

Drawing from classical literature in high-dimensional Bayesian analysis, this study adopts an approach employing priors that induce sparsity in $\beta$ coefficients. For that purpose, we jointly consider a prior $\pi_p$ on the number $s$ of non-zero coefficients in $\beta$ and a Laplace prior on the non-zero coefficients in $\beta$ while setting the other components in $\beta$ to zero:
\begin{equation}
\label{prior_S_beta}
(S, \beta) \mapsto \dfrac{\pi_p(s)}{\binom{\db}{s}} g_S(\beta_S) \delta_0(\beta_{S^c}),
\end{equation}
where $S$ is a subset of $s$ elements in $\{1, \dots, \db\}$ that represents the support of $\beta$, \textit{i.e.} the positions of its non-zero elements, $S^c$ is the complementary subset of zero coefficients in $\beta$, $\beta_S=\{\beta_{\ell} \vert \ell \in S\}$ and $\beta_{S^c}=\{\beta_{\ell} \vert \ell \notin S\}$ are the coefficients of $\beta$ on $S$ and $S^c$ respectively, $\delta_0$ is the Dirac measure at zero on $\mathbb{R}^{\db-s}$ and 
\begin{equation}
\label{eq_loi_Laplace}
    g_S(\beta_S) = \prod_{\ell \in S} \dfrac{\lambda}{2} \exp(- \lambda \vert \beta_{\ell} \vert).
\end{equation}

Concerning the random effects covariance matrix $\Gamma$, a conjugate inverse-Wishart prior is used:

\[\pi(\Gamma) \propto \vert \Gamma \vert^{-\frac{d+r+1}{2}} \exp\left(-\dfrac{1}{2} \text{Tr}(\Sigma \Gamma^{-1})\right),\]
where $\Sigma$ is a positive definite matrix of size $r \times r$,  and $d>r-1$ the degree of freedom. This prior is chosen for a practical matter. Note that, as discussed in \citet{ning_bayesian_2020}, the inverse-Wishart prior may induce sub-optimal posterior contraction rate due to its weaker tail property when $r$ increases to infinity. However, here $r$ is assumed to be fixed so the rate should not be impacted by this property.

\subsection{Assumptions}\label{sec::ass}

We adopt the classical frequentist assumption that the data, consisting of $n$ independent observations coming from $n$ independent subjects $Y^{(n)} = (Y_i)_{1 \leq i \leq n} \in \mathbb{R}^N$, where $N = \sum_{i=1}^n m_i$, were generated from model defined by~\eqref{NLMM} under a given sparse regression parameter $\beta_0$ and a given random effects covariance matrix $\Gamma_0$. The expectation under these true parameters is denoted $\mathbb{E}_{0}$. The support of the true parameter $\beta_0$ is denoted $S_0$ and its cardinal $s_0$. We assume that $\max_{1\leq i \leq n} m_i \leq M_{\text{obs}}$, where $M_{\text{obs}}>0$ is a fixed constant that does not depend on $n$.

\subsubsection{Assumptions on the non-linear model structure}

Assumptions have to be made on the regression function $f$ to obtain posterior contraction. A first natural condition is the Lipschitz assumption, allowing for easy control of the regression function from its inputs.
\begin{assumption}
\label{hyp_lip}
$f : \mathbb{R}^q \times \mathbb{R} \rightarrow \mathbb{R}$ is $K$-Lipschitz with respect to its first component:
\[\quad \forall x, y \in \mathbb{R}^q,  \forall t\in \mathbb{R},  \vert f(x,t) - f(y,t) \vert \leq K \| x - y \|_2\]
\end{assumption}
\begin{remark}
Under Assumption~\ref{hyp_lip}, notice that $f_i : \mathbb{R}^q  \rightarrow \mathbb{R}^{m_i}$ is $\sqrt{K^2 m_i}$-Lipschitz for $\| \cdot \|_2$. 
\end{remark}
As outlined in the introduction, to satisfy the condition of prior mass around the true parameters, they should be confined within a specific subset of the parameter space characterized by bounded norms. Also, it is assumed that $\beta_0$ is sparse and not the zero vector, and that $p$ does not diverge faster than exponential of $n$.

\begin{assumption}
    \label{hyp_beta0}
    The true value $\beta_0$ belongs to
    $$\mathcal{B}_0:=\left\{\beta_0 \in \mathcal{B}\backslash\{0\} : \|\beta_0\|_{\infty} \lesssim \lambda^{-1} \log(p), s_0\log(p)=o(n)\right\},$$ where $\lambda$ is the regularization parameter of the Laplace distribution defined in Equation~\eqref{eq_loi_Laplace} and satisfying Assumption \ref{hyp_lambda} below, and $s_0$ is the true support size. 
\end{assumption}

\begin{assumption}
    \label{hyp_gamma0}
    The true covariance matrix of the random effects $\Gamma_0$ belongs to $$\mathcal{H}_0~:= ~\left\{\Gamma \in \mathcal{H}: 1 \lesssim \rho_{min}(\Gamma) \leq \rho_{max}(\Gamma) \lesssim 1 \right\}, $$ and we denote $\underline{\rho_{\Gamma_0}}>0$ and $\overline{\rho_{\Gamma_0}}>0$ the bounds such that: $\underline{\rho_{\Gamma_0}}~ \leq ~ \rho_{min}(\Gamma_0) \leq \rho_{max}(\Gamma_0) \leq \overline{\rho_{\Gamma_0}}$.
\end{assumption}
Assumption \ref{hyp_beta0} allows that the prior assigns sufficient mass on a Kullback-Leibler neighborhood of $\beta_0$. In the same way, assumption \ref{hyp_gamma0} enables to put sufficient mass around the true parameter $\Gamma_0$ in terms of Frobenius norm. Similar conditions can be found in the work of \citet{ning_bayesian_2020}, \citet{jeong_unified_2021}, and \citet{song2023nearly}. This is in contrast to \citet{castillo_bayesian_2015}'s work where they obtain a result uniformly over the entire parameter space because they have explicit expressions to satisfy this condition directly in their case of univariate regression with known variance. 

\subsubsection{Assumptions on the prior distributions}

The importance of the prior $\pi_p$ lies in its essential role in representing the sparsity of the parameter. The crucial aspect of the prior $\pi_p$ on model dimension is to appropriately reduce the influence of larger models while maintaining sufficient weight for the true one. It is revealed that an exponential decrease effectively fulfills this requirement \citep{castillo_bayesian_2015}. The following assumption is made on $\pi_p$. 
\begin{assumption}[Prior dimension]
\label{hyp_pi_p}
For some constants $A_1$, $A_2$, $A_3$,  $A_4 >0$,  
\[A_1 p^{-A_3} \pi_p(s-1) \leq \pi_p(s) \leq A_2 p^{-A_4} \pi_p(s-1) \quad \text{ ,  for } s=1, \dots, \db.\]
\end{assumption}
Examples of priors satisfying this assumption \ref{hyp_pi_p} are given in \citet{castillo} and \citet{castillo_bayesian_2015}. In fact, this type of prior is more generic than the discrete spike-and-slab prior. Indeed, if each coordinate $\beta_{\ell}$ is modeled as a mixture $(1-w)\delta_0 + w G$, where $G$ follows the Laplace distribution, it can be realized as a prior of the form \eqref{prior_S_beta} by selecting $\pi_p$ as a binomial distribution with parameters $\db$ and $w$. Since $w$ controls the level of sparsity, which is unknown, a classical Bayesian strategy is to put a hyper-prior $Beta(1, (\db)^u)$ with $u>1$. Then, the overall prior satisfies the exponential decay rate \ref{hyp_pi_p}.  
Furthermore, the regularization parameter of the Laplace prior $\lambda$ must be bounded from below and above, as specified in the following assumption. Indeed, an excessively large value of $\lambda$ would shrink non-zero coordinates of $\beta$ towards 0, which is undesirable. Conversely, a too small value of $\lambda$ may introduce false signals in the support, thereby slowing down the posterior contraction rate.
\begin{assumption}
    \label{hyp_lambda}
    The regularization parameter $\lambda$ of the Laplace prior on the non-zero coordinates of $\beta$ satisfies: $$\dfrac{ \|X\|_{*} K'}{L_1p^{L_2}}\leq \lambda \leq \dfrac{L_3 \|X\|_{*} K'}{\sqrt{n}},$$ for some constants $L_1$, $L_2$, $L_3>0$, where $K'=\sqrt{K^2 M_{\text{obs}}}$ and $X = \begin{pmatrix}
X_1 \\
\vdots \\
X_n \\
\end{pmatrix} \in\mathbb{R}^{nq \times \db}$.
\end{assumption}

\begin{remark}
    By defining $V$ as $V = \begin{pmatrix}
V_1^\top \\
\vdots \\
V_n^\top \\
\end{pmatrix} \in\mathbb{R}^{n \times p}$ the covariates matrix, note that $\|X\|_{*}=\|V\|_{*}$.
\end{remark}

Similar condition can be found in \citet{jeong_unified_2021} for example. Note that, since the size of signal in $\beta_0$ is restricted (assumption \ref{hyp_beta0}), the Laplace density is not required to achieve the posterior contraction rates in Theorem \ref{Th_beta}. Other slab densities with similar tail properties, such as the Gaussian slab, can also be used with appropriate adjustments for the true signal size (see \citet{jeong_posterior_2021,jeong_unified_2021} for more details).

\subsubsection{Assumptions about the experimental design}

For $\Gamma_1, \Gamma_2 \in \mathcal{H}$, we define the pseudo distance $$d_n^2(\Gamma_1, \Gamma_2)=\dfrac{1}{n} \sum_{i=1}^n \| \Delta_{\Gamma_1,i}-\Delta_{\Gamma_2,i} \|_F^2 = \dfrac{1}{n} \sum_{i=1}^n \| Z_i(\Gamma_1-\Gamma_2)Z_i^\top \|_F^2,$$ where $\Delta_{\Gamma,i}$ is defined in equation \eqref{NLMM_compact}. The following three assumptions on the model \ref{hyp_ni_sup_q}, \ref{hyp_rg_Zi}, \ref{hyp_Zi_sp} enable to control the maximum Frobenius norm of the difference between covariance matrices from the average Frobenius norm:
$$\max_i \|\Delta_{\Gamma_1,i}-\Delta_{\Gamma_2,i} \|_F^2 \lesssim  \|\Gamma_1-\Gamma_2 \|_F^2 \lesssim d_n^2(\Gamma_1,\Gamma_2)=\frac{1}{n}\sum_{i=1}^n \|\Delta_{\Gamma_1,i}-\Delta_{\Gamma_2,i}\|_F^2.$$
This point is demonstrated in Appendix \ref{appB}.

\begin{assumption}
    \label{hyp_ni_sup_q}
    The quantity $\sum_{i=1}^n \mathds{1}_{m_i \geq \dg}$ satisfies: $\sum_{i=1}^n \mathds{1}_{m_i \geq \dg} \asymp n$.
\end{assumption}
Assumption \ref{hyp_ni_sup_q} means that the number of individuals $i$ such as the number of observations $m_i$ is greater than the number of random effects $\dg$ is of the order of $n$, that is $m_i$ is probably greater than $\dg$, which seems statistically reasonable to be able to estimate $\Gamma$ and $\beta$. This assumption contributes to the model's identifiability.

\begin{assumption}
    \label{hyp_rg_Zi}
    For each $1\leq i \leq n$ such that $m_i \geq \dg$, $Z_i$ is of full rank, \textit{i.e.} $\min_i \left\{ \rho_{min}^{1/2}(Z_i^\top Z_i) : m_i \geq \dg \right\} \gtrsim  1$. We denote by $\underline{\rho_Z}$ the bound: $$\min _i \left\{ \rho_{min}^{1/2}(Z_i^\top Z_i) : m_i \geq \dg \right\} \geq \underline{\rho_Z}.$$
\end{assumption}

\begin{assumption}
    \label{hyp_Zi_sp}
    For each $1 \leq i \leq n$, the maximum of $\| Z_i\| _{sp}$ is bounded, \textit{i.e.} $\max _i \| Z_i\| _{sp}\lesssim 1$. We denote by $\overline{\rho_Z}$ the bound: $$\max _i \| Z_i\| _{sp} \leq \overline{\rho_Z}.$$
\end{assumption}
Similar assumptions can be found in Theorem $10$ of \citet{jeong_unified_2021} for the linear mixed-effects model.

\section{Posterior contraction results}
\label{sec_results}

In this section, we provide results on posterior contraction rates in sparse non-linear marginal mixed-effects model under suitable assumptions presented in Section~\ref{sec::ass}. To achieve this, we first analyze a dimensionality property of the support of $\beta$. Then, we determine how quickly the posterior contracts based on the average Rényi divergence. Finally, we use this information about Rényi contraction to establish the rates for the parameters relative to more practical metrics.

\subsection{Support size control}
\label{sec_supp_size}

First, it is essential to examine the support size of $\beta$ in order to then focus on models of relatively small sizes. The following theorem shows that the posterior distribution tends to concentrate on models of relatively small sizes, not much larger than the true one.
\begin{theorem}[Effective dimension]
\label{Th_dim}
In model \eqref{NLMM}, with prior specifications outlined in Section \ref{sec_prior}, and assuming the validity of previous assumptions \ref{hyp_lip}-\ref{hyp_Zi_sp}, there exists a constant $C_1>0$ such that the following convergence holds:
\[ \underset{\beta_0 \in \mathcal{B}_0, \Gamma_0 \in \mathcal{H}_0}{\sup} \mathbb{E}_{0}\left[\Pi\left(\beta : |S_{\beta}|> C_1 s_0 \bigg| Y^{(n)}\right)\right] \underset{n \rightarrow \infty}{\longrightarrow} 0.\]
\end{theorem}
Proof of this theorem is provided in Section \ref{proof_Th1}. The derivation of the posterior contraction rate heavily relies on a technical lemma which provides a lower bound for the denominator of the posterior distribution with probability tending to $1$, see Lemma \ref{Lem_denom} in Section \ref{proof_Th1}. More precisely, this lemma is employed in deriving our main results on effective dimension and posterior contraction rates, as outlined in Theorems \ref{Th_dim} and \ref{Th_Renyi}.

\subsection{Posterior contraction rates}

As discussed in the introduction, the classical approach for determining posterior contraction rates encounters limitations when dealing with the unknown nature of the random effects covariance matrix. Indeed, this approach based on the average squared Hellinger distance faces inadequacies in obtaining rates in terms of the Euclidean norm for the parameters in this context. Specifically, the issue arises from the fact that establishing proximity using the average squared Hellinger distance between multivariate normal densities with individual-specific mean and an unknown covariance does not guarantee average proximity in terms of the Euclidean distance for the mean parameters in these densities. To overcome this challenge, the proposed solution is a direct utilization of the average Rényi divergence of order 1/2 (see Definition \ref{def_Rényi}). This approach is highlighted for its high manageability in the context of multivariate normal distributions and its ability to ensure closeness in terms of the desired Euclidean distance. Examples of the application of this theory can be found in the works of \citet{ning_bayesian_2020} and \citet{jeong_unified_2021}, further supporting the efficacy of the average Rényi divergence in overcoming the limitations associated with the unknown covariance matrix for random effects.

\begin{definition}
\label{def_Rényi}
    For two $n$-variates densities $f=\prod_{i=1}^n f_i$ and $g=\prod_{i=1}^n g_i$ of independent variables, the average Rényi divergence (of order 1/2) is defined by:  $$R_n(f,g)=-\dfrac{1}{n}\sum_{i=1}^n \log \left( \int \sqrt{f_i g_i} \right)$$
\end{definition}
Based on the result of Theorem \ref{Th_dim}, the following theorem establishes the rate of contraction of the posterior distribution towards the truth with respect to the average Rényi divergence.

\begin{theorem}[Contraction rate, Rényi]
\label{Th_Renyi}
In model \eqref{NLMM}, with prior specifications outlined in Section \ref{sec_prior}, we denote by $p_{\beta, \Gamma}=\prod_{i=1}^n p_{\beta, \Gamma,i}$ the joint density, with $p_{\beta, \Gamma,i}$ representing the density of the $i$th observation vector $y_i$, and $p_0$ representing the true joint density. Assuming the previous assumptions \ref{hyp_lip}-\ref{hyp_Zi_sp} hold, then there exists a constant $C_2>0$ such that:
$$\underset{\beta_0 \in \mathcal{B}_0, \Gamma_0 \in \mathcal{H}_0}{\sup} \mathbb{E}_{0}\left[\Pi\left((\beta,\Gamma) \in \mathcal{B} \times \mathcal{H} : R_n(p_{\beta, \Gamma},p_0) > C_2 \frac{s_0 \log(p)}{n} \bigg| Y^{(n)}\right)\right] \underset{n \rightarrow \infty}{\longrightarrow} 0.$$
\end{theorem}
The proof can be found in Section \ref{Proof_Th2}. This proof is based on the general theory of posterior contraction rate of \citet{ghosal_convergence_2000,ghosal_convergence_2007,ghosal2017fundamentals}, which relies on the construction and existence of exponentially powerful tests (see also \citet{castillo2024nonparam} for more details).

While Theorem \ref{Th_Renyi} provides a fundamental result on posterior contraction, it does not offer precise interpretations for the parameters $\beta$ and $\Gamma$. The following theorem relies on the form of the average Rényi divergence to obtain more concrete contraction rates. Specifically, it demonstrates that the posterior distribution of the prediction term and $\Gamma$ contracts towards their true respective values at certain rates, relative to metrics more easily understandable than the average Rényi divergence.

\begin{theorem}[Recovery]
\label{Th_recovery}
    In model \eqref{NLMM}, with prior specifications outlined in Section \ref{sec_prior}, and assuming Assumptions \ref{hyp_lip}-\ref{hyp_Zi_sp} , then there exist constants $C_3,C_4,C_5>0$ such that:
    $$\underset{\beta_0 \in \mathcal{B}_0, \Gamma_0 \in \mathcal{H}_0}{\sup} \mathbb{E}_{0}\left[\Pi\left(\Gamma : d_n(\Gamma,\Gamma_0) > C_3 \sqrt{\frac{s_0 \log(p)}{n}} \bigg| Y^{(n)}\right)\right] \underset{n \rightarrow \infty}{\longrightarrow} 0,$$
    $$\underset{\beta_0 \in \mathcal{B}_0, \Gamma_0 \in \mathcal{H}_0}{\sup} \mathbb{E}_{0}\left[\Pi\left(\Gamma : \|\Gamma - \Gamma_0\|_F > C_4 \sqrt{\frac{s_0 \log(p)}{n}} \bigg| Y^{(n)}\right)\right] \underset{n \rightarrow \infty}{\longrightarrow} 0,$$
    $$\underset{\beta_0 \in \mathcal{B}_0, \Gamma_0 \in \mathcal{H}_0}{\sup} \mathbb{E}_{0}\left[\Pi\left(\beta : \mathcal{P}_n > C_5 \sqrt{\frac{s_0 \log(p)}{n}} \bigg| Y^{(n)}\right)\right] \underset{n \rightarrow \infty}{\longrightarrow} 0,$$
    where $\mathcal{P}_n=\sqrt{\frac{1}{n}\sum_{i=1}^n \|f_i(X_i\beta)-f_i(X_i\beta_0)\|_2^2}$.
\end{theorem}
The proof can be found in Section \ref{Proof_Th3}. By comparing our theorem to \citet{castillo_bayesian_2015}'s results in Bayesian, or \citet{buhlmann2011statistics}'s results in frequentist framework, in simple linear regression, it can be observed that the same rates are achieved for the prediction term. For the covariance term, the rate obtained in Theorem \ref{Th_recovery} coincides with that obtained for linear regression with nuisance parameters by \citet{jeong_unified_2021}. However, \cite{ning_bayesian_2020} achieve a sharper rate of the form $\sqrt{\dfrac{r^2\log(n)}{n}}$. In our case, the rate we establish reflects technical limitations of our proof approach.

The last theorem gives precise interpretations of the posterior contraction result for the parameter $\beta$. The posterior contraction rates with respect to more concrete metrics are derived based on an additional condition, summarized by Assumption \ref{hyp_identifiability} .

\begin{assumption}
    \label{hyp_identifiability}
    For some $\eta>0$, there exists a constant $L$ such that for all $\delta>0$ small enough, there exists $n_0\in\mathbb{N}$ such that for all  $n \geq n_0$:
    $$ \sup\Big\{ \frac{1}{n} \|X(\beta-\beta_0)\|_2^2 : \beta \in \mathcal{B}, \frac{1}{n} \sum_{i=1}^n \sum_{j=1}^{m_i} |f(X_i\beta,t_{ij})-f(X_i\beta_0,t_{ij})|^2 \leq \delta \Big\} \leq L \delta^{\eta}.$$ 
\end{assumption}

This assumption, as well as Assumption \ref{hyp_ni_sup_q}, is a sufficient condition for the identifiability of the model and allows to derive posterior contraction rate for $\beta$ from the third assertion of Theorem \ref{Th_recovery}. This condition is inspired by the so-called "stability estimate" condition used in the context of nonlinear inverse problems \cite{nickl2023bayesian}. To ensure the identifiability of the parameter $\beta$, a kind of assumption of "local invertibility" for the Gram matrix $X^\top X$ is also required. For this purpose, we define the following compatibility numbers drawing from the literature \citep{castillo_bayesian_2015}.

\begin{definition}
\label{def_phi2}
    For all $s>0$, the smallest scaled singular value of dimension $s$ is defined as: $$\phi_2(s)= \inf_{\beta : 1 \leq s_{\beta}\leq s }\dfrac{\|X\beta\|_2}{\|X\|_{*}\|\beta\|_2}.$$
\end{definition}

\begin{definition}
\label{def_phi1}
    For all $s>0$, the uniform compatibility number in dimension $s$ is defined as: $$\phi_1(s)= \inf_{\beta : 1 \leq s_{\beta}\leq s }\dfrac{\|X\beta\|_2 \sqrt{s_{\beta}}}{\|X\|_{*}\|\beta\|_1}.$$
\end{definition}

\begin{theorem}[Posterior contraction rate for $\beta$]
\label{Th_beta}
    In model \eqref{NLMM}, with prior specifications outlined in Section \ref{sec_prior}, and assuming Assumptions \ref{hyp_lip}-\ref{hyp_identifiability}, then there exist constants $C_6, C_7, C_8>0$ such that:
    $$\underset{\beta_0 \in \mathcal{B}_0, \Gamma_0 \in \mathcal{H}_0}{\sup} \mathbb{E}_{0}\left[\Pi\left(\beta : \|X(\beta-\beta_0)\|_2 > C_6  \dfrac{(s_0\log(p))^{\eta/2}}{n^{(\eta-1)/2}} \bigg| Y^{(n)}\right)\right] \underset{n \rightarrow \infty}{\longrightarrow} 0,$$
    $$\underset{\beta_0 \in \mathcal{B}_0, \Gamma_0 \in \mathcal{H}_0}{\sup} \mathbb{E}_{0}\left[\Pi\left(\beta : \|\beta-\beta_0\|_2 > C_7 \dfrac{n^{(1-\eta)/2}(s_0\log(p))^{\eta/2}}{\|X\|_{*} \phi_2((C_1+1)s_0)} \bigg| Y^{(n)}\right)\right] \underset{n \rightarrow \infty}{\longrightarrow} 0, $$
    $$\underset{\beta_0 \in \mathcal{B}_0, \Gamma_0 \in \mathcal{H}_0}{\sup} \mathbb{E}_{0}\left[\Pi\left(\beta : \|\beta-\beta_0\|_1 > C_8 \dfrac{n^{(1-\eta)/2} s_0^{(\eta+1)/2} \log(p)^{\eta/2}}{\|X\|_{*} \phi_1((C_1+1)s_0)} \bigg| Y^{(n)}\right)\right] \underset{n \rightarrow \infty}{\longrightarrow} 0. $$
\end{theorem}
The proof can be found in Section \ref{Proof_Th4}. Since the compatibility numbers can be bounded away from zero under some specific conditions on the design matrix (see Example 7 of \citet{castillo_bayesian_2015} for further discussion), they can be removed from the rates. Moreover, if we reasonably assume that for each $1 \leq i \leq n$, the maximum of $\| X_i\|_{*}$ is bounded, \textit{i.e.}  $\max _i \| X_i\| _{*} \asymp 1$ (commonly satisfied in practical scenarios), we have that $\|X\|_{*} \asymp \sqrt{n}$. Also, these rates coincide with those obtained by \citet{jeong_posterior_2021} or \citet{jeong_unified_2021}, respectively, for generalized linear models and linear regression with nuisance parameters, if Assumption \ref{hyp_identifiability} is satisfied for $\eta=1$. 

To the best of our knowledge, the only established theoretical results on frequentist variable selection in high-dimensional settings are primarily due to \citet{schelldorfer2011}, who uses an $\ell_1$ penalty and focuses on linear mixed-effects models. Our results are particularly noteworthy, as they extend beyond the linear case to include nonlinear models, subject to certain regularity and stability conditions. There are clear similarities between the assumptions in \cite{schelldorfer2011} and ours: parameters evolving in a bounded space, a well-posed random effects model with an invertible variance-covariance matrix, a sparse fixed effects vector, and eigenvalue conditions on $Z_i^T Z_i$. In terms of oracle optimality, Schelldorfer and coauthors \cite{schelldorfer2011} establish a bound that shows that the sparser the model (\textit{ie} for small $s_0$), the better the performance, provided that the number of individuals is sufficiently large relative to $\log p$ to approach the oracle. Our posterior contraction rate exhibits a similar behavior, reinforcing the analogy between the two results. Moreover, if Assumption \ref{hyp_identifiability} holds for $\eta=1$, our rate is better than theirs. Note that \cite{ghosh2018non} propose an extension of these results to more general non-concave penalties. However, while our Bayesian approach enables us to derive insights into the estimation of the covariance matrix of the random effects, Schelldorfer’s work does not provide results on this matrix. Recently, \cite{zhang2023selection} studied the asymptotic oracle properties of non-concave penalized quasi-likelihood estimator under some conditions with a proxy covariance matrix in a high-dimensional generalized linear mixed models.

\subsection{Example of non-linear marginal mixed-effects model}

Let us revisit the example of the logistic growth function (see section \ref{ss_sec_model}) with a single individual-specific effect ($q=1$), as described by the following equation: $$f(\varphi,t)=\dfrac{A}{1+e^{(t-t_0)\varphi}},$$
with $t\in [0,T], A>0, t_0>0, T>0$. We can check that model defined by~\eqref{NLMM} with this non-linear function $f$ satisfies the assumptions \ref{hyp_lip} and \ref{hyp_identifiability} set out above for obtaining posterior contraction rates. First, note that here $\beta$ is a vector of dimension $p$ and $X_i=V_i^\top$ is a row vector of dimension $p$.

By studying the derivative of $f$ with respect to $\varphi$, it is straightforward to observe that $f$ is Lipschitz with respect to its first argument, with a Lipschitz constant of the form $C(A,T,t_0)/4$, and therefore satisfies Assumption \ref{hyp_lip}.

Moreover, if $\varphi$ belongs to a compact set, we have that $\bigg| \dfrac{\partial f(\varphi,t)}{\partial \varphi} \bigg| \geq \alpha >0$ for all $t \in [0,T]$. Thus, let us further suppose that the covariates are well controlled, and more specifically that $\max_i \|V_i\|_1 \lesssim 1$, and that $\beta$ lies in a compact set such that $\|\beta\|_{\infty} \lesssim 1$. Then, by the mean value theorem, we have:
$$ \frac{1}{n} \sum_{i=1}^n \sum_{j=1}^{m_i} |f(X_i\beta,t_{ij})-f(X_i\beta_0,t_{ij})|^2 \geq \dfrac{\alpha}{n} \sum_{i=1}^n |X_i(\beta-\beta_0)|^2 = \dfrac{\alpha}{n} \| X(\beta-\beta_0)\|_2^2,$$ 
which proves Assumption \ref{hyp_identifiability} with $\eta=1$.

\section{Conclusion}
\label{sec_conclusion}

In this work, we have established a posterior contraction result for a non-linear mixed-effects model within the high-dimensional setting. The strength of our approach lies in the fact that it applies to a broader class of mixed-effects models beyond just linear ones, thus making it relevant to a wide range of practical applications. The proof of this result, however, involves overcoming several significant technical hurdles.

Furthermore, the computational challenges associated with variable selection in high-dimensional non-linear mixed-effects models have been addressed by \citet{naveau_bayesian_2023} for a Gaussian spike-and-slab prior. Their work offers valuable insights into the algorithmic strategies necessary for efficient inference in such complex models. In the context of a discrete spike-and-slab prior, or more generally priors of the form \eqref{prior_S_beta}, \cite{geweke1996variable,george_approaches_1997} employed Gibbs sampling within the framework of linear regression models. However, for more complex models like the one considered in this article, and particularly in high-dimensional environments, efficient sampling from complex posterior distributions remains a major challenge.

In addition, we focus here on the inverse-Wishart prior for the covariance matrix, as this is a common and practical choice in applied contexts, but it seems to give sub-optimal rates in this proof case. It would therefore be interesting to study in detail the precise asymptotic rate for $\Gamma$. Moreover, variable selection properties are of primary importance in practice and rely on posterior contraction results obtained in this paper \citep{castillo_bayesian_2015,jeong_unified_2021}. The study of variable selection is beyond the scope of this article and is naturally left for a future research work.

\section{Proofs of main theorems}
\label{sec_proofs}

In this section, proofs of the main theorems are provided following the general structure of those presented in \cite{jeong_unified_2021}, which is itself based on the general theory of \cite{ghosal_convergence_2000,ghosal_convergence_2007,ghosal2017fundamentals}. However, significant adaptations were required due to the fundamental differences in our model. Indeed, the non-linear nature of our framework, as opposed to the linear setting in \cite{jeong_unified_2021}, necessitated new assumptions and a careful reworking of several crucial arguments. In particular, the originality of the proof of Theorem \ref{Th_beta}, which consists of deriving a posterior contraction rate on $\beta$ vector from the contraction rate on the prediction term, comes from non-linear statistical inverse problem tools \cite{nickl2023bayesian}.

First, additional notation used for the proofs is introduced. Let $\Lambda_n(\beta,\Gamma)=\prod_{i=1}^n p_{\beta,\Gamma,i}/p_{0,i}$ be the likelihood ratio of $p_{\beta,\Gamma}=\prod_{i=1}^n p_{\beta,\Gamma,i}$, where $p_{\beta,\Gamma,i}$ is the density of the $i$-th observation vector $y_i$, and $p_{0}=\prod_{i=1}^n p_{0,i}= \prod_{i=1}^n p_{\beta_0,\Gamma_0,i}$ the density with the true parameters $\beta_0$ and $\Gamma_0$.
For two densities $g$ and $h$, let $K(g,h)= \int g(x) \log \left(g(x)/h(x) \right) dx$ the Kullback-Leibler divergence, and $V(g,h)~=~\int g(x) \left| \log\left(g(x)/h(x)\right) - K(g,h) \right|^2 dx$ the Kullback-Leibler variation. 

\subsection{Proof of Theorem~\ref{Th_dim}}
\label{proof_Th1}

The proof of Theorem~\ref{Th_dim} is based on the following technical lemma which provides a lower bound for the denominator of the posterior distribution, with the probability tending to $1$.

\begin{lemma}
    \label{Lem_denom}
    Suppose that Assumptions~\ref{hyp_lip}-\ref{hyp_Zi_sp} are satisfied. Then, there exists a positive constant $M$ such that \begin{equation}
    \label{eq_Lem1}
        \mathbb{P}_0 \left(\int \Lambda_n(\beta,\Gamma) d\Pi(\beta,\Gamma) \geq \pi_p(s_0) e^{-M(s_0 \log(p) + \log(n))} \right) \longrightarrow 1,
    \end{equation}
    when $n$ tends to infinity.
\end{lemma}
This lemma is demonstrated in Appendix \ref{appA}.

\begin{proof}[Proof of Theorem~\ref{Th_dim}]
    Let $B=\{(\beta,\Gamma) \in \mathcal{B} \times \mathcal{H} : |S_{\beta}| > \Tilde{s}\}$, with an integer $\Tilde{s} \geq s_0$. First, by the Bayes formula: \begin{equation}
        \label{formule_Bayes}
        \Pi(B| Y^{(n)})=\dfrac{\int_B \Lambda_n(\beta,\Gamma) d\Pi(\beta,\Gamma)}{\int \Lambda_n(\beta,\Gamma) d\Pi(\beta,\Gamma)}.
    \end{equation}
    Let us prove that $\mathbb{E}_{0}\left[ \Pi(B| Y^{(n)})  \right]$ tends to $0$ as $n$ tends to infinity uniformly for $\beta_0 \in \mathcal{B}_0$ and $\Gamma_0 \in \mathcal{H}_0$, and choose a suitable $\Tilde{s}$. Let $\mathcal{A}_n$ be the event that appears in Equation~\eqref{eq_Lem1}. We can write
    \begin{equation}
    \label{decomp_E0}
    \mathbb{E}_{0}\left[\Pi\left(B | Y^{(n)}\right)\right]=\mathbb{E}_{0}\left[\Pi\left(B | Y^{(n)}\right)\mathds{1}_{\mathcal{A}_n}\right] + \mathbb{E}_{0}\left[\Pi\left(B | Y^{(n)}\right)\mathds{1}_{\mathcal{A}_n^c}\right].
    \end{equation}
    where the second term tends to $0$ by using Lemma \ref{Lem_denom}.
    
    Concerning the first term, by definition of $\mathcal{A}_n$, we have that
        \begin{align*}
        \mathbb{E}_{0} \left[\Pi\left(B | Y^{(n)}\right)\mathds{1}_{\mathcal{A}_n}\right] &= \mathbb{E}_{0}\left[\dfrac{\int_B \Lambda_n(\beta,\Gamma) d\Pi(\beta,\Gamma)}{\int \Lambda_n(\beta,\Gamma) d\Pi(\beta,\Gamma)}\mathds{1}_{\mathcal{A}_n}\right] \\
        & \leq \mathbb{E}_{0}\left[\int_B \Lambda_n(\beta,\Gamma) d\Pi(\beta,\Gamma) \pi_p(s_0)^{-1} e^{M(s_0 \log(p) + \log(n))}  \mathds{1}_{\mathcal{A}_n}\right] \\
        & \leq \pi_p(s_0)^{-1} e^{M(s_0 \log(p) + \log(n))} \mathbb{E}_{0}\left[\int_B \Lambda_n(\beta,\Gamma) d\Pi(\beta,\Gamma)  \mathds{1}_{\mathcal{A}_n} \right] \\
    \end{align*}
        Now, we get that
    \begin{align*}
    \mathbb{E}_{0}\left[\int_B \Lambda_n(\beta,\Gamma) d\Pi(\beta,\Gamma)  \mathds{1}_{\mathcal{A}_n} \right] &\leq \mathbb{E}_{0}\left[\int_B \dfrac{p_{\beta,\Gamma}(y)}{p_0(y)} d\Pi(\beta,\Gamma) \right] \\
    &= \int \int_B p_{\beta,\Gamma}(y) d\Pi(\beta,\Gamma) dy \\
    &= \Pi(B)
    \end{align*}
    using Fubini-Tonelli theorem and since $p_{\beta,\Gamma}$ is a density.
        Thus,
    \[\mathbb{E}_{0}\left[\Pi\left(B | Y^{(n)}\right)\mathds{1}_{\mathcal{A}_n}\right] \leq \pi_p(s_0)^{-1} \exp{\left\{M(s_0 \log(p) + \log(n))\right\}}\Pi(B), \]
     and by Assumption \ref{hyp_pi_p},
    \begin{align*}
    \Pi(B) &= \Pi( |S_{\beta}| > \Tilde{s}) = \sum_{s=\Tilde{s}+1}^{\db} \pi_p(s) \\
    & \leq \pi_p(s_0) \sum_{s=\Tilde{s}+1}^{\db} \left( A_2 p^{-A_4}\right)^{s-s_0} \\
    &= \pi_p(s_0)  \left( A_2 p^{-A_4}\right)^{\Tilde{s}+1-s_0} \sum_{k=0}^{\db-\Tilde{s}-1} \left( A_2 p^{-A_4}\right)^{k} \\
    &\leq \pi_p(s_0)  \left( A_2 p^{-A_4}\right)^{\Tilde{s}+1-s_0} \dfrac{1}{1-A_2 p^{-A_4}},
    \end{align*}
    for $p$ large enough to ensure that $A_2 p^{-A_4}<1$. Thus finally we have
        \begin{align*}
    \mathbb{E}_{0}\left[\Pi\left(B | Y^{(n)}\right)\mathds{1}_{\mathcal{A}_n}\right]& \leq \pi_p(s_0)^{-1} \exp{\left\{M(s_0 \log(p) + \log(n))\right\}}\Pi(B) \\
    &\leq  \exp{\left\{M(s_0 \log(p) + \log(n)) + (\Tilde{s}+1-s_0)\log(A_2 p^{-A_4}) \right\}} \dfrac{1}{1-A_2 p^{-A_4}} \\
    &= \exp{\left\{ \log(p) \left(M s_0+ M \dfrac{\log(n)}{\log(p)} -A_4(\Tilde{s}+1-s_0) \right)  + (\Tilde{s}+1-s_0)\log(A_2) \right\}} \times \\
    & \hspace{1cm}\dfrac{1}{1-A_2 p^{-A_4}} \\
    \end{align*}
    where $\log(n)/\log(p)\leq 1$ as $p>n$. Thus,  as $ (1-A_2 p^{-A_4})^{-1} $ tends to 1 when $n \rightarrow \infty$, we choose $\Tilde{s}$ as the largest integer that is smaller than $C_1 s_0$ (such as $\Tilde{s}+1> C_1 s_0$), for some constant $C_1$ large enough to have $Ms_0+M-A_4(C_1 s_0-s_0)<0$, and then we have that
    \begin{align*}
    \mathbb{E}_{0}\left[\Pi\left(B | Y^{(n)}\right)\mathds{1}_{\mathcal{A}_n}\right]
    & \leq  \exp{\big\{ \log(p) \left(M s_0+ M -A_4(C_1 s_0-s_0) \right)  + (\Tilde{s}+1-s_0)\log(A_2) \big\}} \dfrac{1}{1-A_2 p^{-A_4}}
    \end{align*}
where the term on the right tends to zero when $n$ goes to infinity.
Finally,  by Equation~\eqref{decomp_E0}, we conclude that $\mathbb{E}_{0}\left[ \Pi(B| Y^{(n)})  \right] \longrightarrow 0$, for this well-chosen $\Tilde{s}$. Thus, we have also that $\mathbb{E}_{0}\left[\Pi\left(\beta : |S_{\beta}|> C_1 s_0 \bigg| Y^{(n)}\right)\right] \longrightarrow 0$, which concludes the proof of the theorem. 
\end{proof}

\subsection{Proof of Theorem \ref{Th_Renyi}}
\label{Proof_Th2}

\begin{proof}[Proof of Theorem \ref{Th_Renyi}]

Let $\mathcal{B}_n = \{ \beta \in \mathcal{B} | s_{\beta} \leq C_1 s_0 \}$, $R_n^{*}(\beta, \Gamma)=R_n(p_{\beta, \Gamma},p_0)$ and $\epsilon_n = \sqrt{s_0 \log(p) / n}$.

\begin{align*}
    \mathbb{E}_{0}\Bigg[\Pi\Bigg((\beta,\Gamma) \in \mathcal{B} \times \mathcal{H} &: R_n^{*}(\beta, \Gamma) > C_2 \epsilon_n^2 \bigg|Y^{(n)}\Bigg)\Bigg] \\
    & \leq \mathbb{E}_{0}\left[\Pi\left((\beta,\Gamma) \in \mathcal{B}_n \times \mathcal{H} : R_n^{*}(\beta, \Gamma) > C_2 \epsilon_n^2 \bigg| Y^{(n)}\right)\right] + \mathbb{E}_{0}\left[\Pi\left(\mathcal{B}_n^{c} | Y^{(n)}\right)\right]
\end{align*}
where the second term tends to $0$ when $n$ goes to infinity by Theorem~\ref{Th_dim}.

Therefore, given $D~= ~\left \{ (\beta,\Gamma) \in \mathcal{B}_n \times \mathcal{H} : R_n^{*}(\beta, \Gamma) > C_2 \epsilon_n^2 \right \}$, proving Theorem \ref{Th_Renyi} consists in showing that $\mathbb{E}_{0}\left[\Pi\left(D | Y^{(n)}\right)\right]$ goes to $0$ as $n$ tends to infinity uniformly for $\beta_0 \in \mathcal{B}_0$ and $\Gamma_0 \in \mathcal{H}_0$.

This proof is based on the construction and existence of exponentially powerful tests to show contraction rates of posterior distributions (see \citet{ghosal_convergence_2000,ghosal2017fundamentals} for more details). More precisely, we want to construct a test $\varphi_n$ such that on an appropriate sieve $\mathcal{B}_n^{*} \times \mathcal{H}_n \subset \mathcal{B}_n \times \mathcal{H}$ we have, for some constants $M_1$, $M_2>0$:
\begin{equation}
    \label{test_phi}
    \mathbb{E}_0[\varphi_n] \lesssim e^{-M_1 n \epsilon_n^2} \quad \text{ , } \quad \sup_{(\beta, \Gamma) \in \mathcal{B}_n^{*} \times \mathcal{H}_n : R_n^{*}(\beta, \Gamma) > C_2 \epsilon_n^2 } \mathbb{E}_{(\beta, \Gamma)}[1-\varphi_n] \leq e^{-M_2 n \epsilon_n^2}
\end{equation}
where the sieve $\mathcal{B}_n^{*} \times \mathcal{H}_n$ shall satisfy that the prior mass of $\mathcal{B}_n \backslash \mathcal{B}_n^{*}$ and $\mathcal{H} \backslash \mathcal{H}_n$ decreases rapidly enough to balance the denominator of the posterior. Indeed, assuming that we have constructed such a test, then, for $\mathcal{A}_n$ the event that appears in Equation~\eqref{eq_Lem1}:

\begin{align*}
    \mathbb{E}_{0}\left[\Pi\left(D | Y^{(n)}\right)\right] &= \mathbb{E}_{0}\left[\Pi\left(D | Y^{(n)}\right) \mathds{1}_{\mathcal{A}_n}\right] + \mathbb{E}_{0}\left[\Pi\left(D | Y^{(n)}\right)\mathds{1}_{\mathcal{A}_n^c}\right] \\
    &= \mathbb{E}_{0}\left[\Pi\left(D | Y^{(n)}\right) \mathds{1}_{\mathcal{A}_n} (1 - \varphi_n) + \Pi\left(D | Y^{(n)}\right) \mathds{1}_{\mathcal{A}_n} \varphi_n\right] + \mathbb{E}_{0}\left[\Pi\left(D | Y^{(n)}\right)\mathds{1}_{\mathcal{A}_n^c}\right] \\
    &\leq \mathbb{E}_{0}\left[\Pi\left(D | Y^{(n)}\right) \mathds{1}_{\mathcal{A}_n} (1 - \varphi_n)\right]  + \mathbb{E}_{0}\left[\varphi_n\right] + \mathbb{P}_{0}\left(\mathcal{A}_n^c\right)  
\end{align*}
where by construction of $\varphi_n$, $\mathbb{E}_{0}\left[\varphi_n\right] \underset{n\rightarrow \infty}{\longrightarrow 0}$, and $\mathbb{P}_{0}\left(\mathcal{A}_n^c\right) \underset{n\rightarrow \infty}{\longrightarrow 0}$ by Lemma \ref{Lem_denom}.

Now for the first term, by the Bayes formula \eqref{formule_Bayes}, we have that

\begin{align*}
    \mathbb{E}_{0}\left[\Pi\left(D | Y^{(n)}\right) \mathds{1}_{\mathcal{A}_n} (1 - \varphi_n)\right] &= \mathbb{E}_{0}\left[\dfrac{\int_D \Lambda_n(\beta,\Gamma) d\Pi(\beta,\Gamma)}{\int \Lambda_n(\beta,\Gamma) d\Pi(\beta,\Gamma)} \mathds{1}_{\mathcal{A}_n} (1 - \varphi_n)\right] \\
    &\leq  \mathbb{E}_{0}\left[\int_D \Lambda_n(\beta,\Gamma) d\Pi(\beta,\Gamma) \pi_p(s_0)^{-1} e^{M(s_0 \log(p) + \log(n))} (1 - \varphi_n)\right]
\end{align*}

But, grant Assumption~\ref{hyp_pi_p}, we have that: $\pi_p(s_0)^{-1} \leq A_1^{-1} p^{A_3} \pi_p(s_0-1)^{-1}$ and by iteration $$-\log(\pi_p(s_0)) \lesssim s_0 \log(p) - \log(\pi_p(0)) \lesssim s_0 \log(p)$$ 
since $1 = \sum_{s=1}^p \pi_p(s) \leq \sum_{s=1}^p (A_2p^{-A_4})^s \pi_p(0) \lesssim \pi_p(0)$ by assumption \ref{hyp_pi_p}. Thus, for a constant $C$ large enough, $\pi_p(s_0)^{-1} e^{M(s_0 \log(p) + \log(n))} \leq e^{C s_0 \log(p)}=e^{C n \epsilon_n^2}$, since $\log(n) \lesssim s_0 \log(p)$. So, by using the Fubini-Tonelli theorem,

\begin{align*}
    \mathbb{E}_{0}\left[\Pi\left(D | Y^{(n)}\right) \mathds{1}_{\mathcal{A}_n} (1 - \varphi_n)\right] &\leq \int_D \mathbb{E}_{(\beta, \Gamma)}\left[1 - \varphi_n \right] d\Pi(\beta,\Gamma) \times  e^{C n \epsilon_n^2} \\
    &\leq \left(\int_{D \cap (\mathcal{B}_n^{*} \times \mathcal{H}_n)} \mathbb{E}_{(\beta, \Gamma)}\left[1 - \varphi_n \right] d\Pi(\beta,\Gamma) + \Pi(\mathcal{B}_n \backslash\mathcal{B}_n^{*}) + \Pi(\mathcal{H} \backslash \mathcal{H}_n) \right) \times  e^{C n \epsilon_n^2} \\
    &\leq \left(\sup_{(\beta,\Gamma) \in D \cap (\mathcal{B}_n^{*} \times \mathcal{H}_n)} \left\{\mathbb{E}_{(\beta, \Gamma)}\left[1 - \varphi_n \right] \right\}+ \Pi(\mathcal{B}_n \backslash\mathcal{B}_n^{*}) + \Pi(\mathcal{H} \backslash \mathcal{H}_n) \right) \times  e^{C n \epsilon_n^2} \\
    & \leq \left(e^{-M_2 n \epsilon_n^2} + \Pi(\mathcal{B}_n \backslash\mathcal{B}_n^{*}) + \Pi(\mathcal{H} \backslash \mathcal{H}_n) \right) \times  e^{C n \epsilon_n^2}
\end{align*}
by construction of $\varphi_n$, equation \eqref{test_phi}. Then for $M_2$ large enough and by the condition on the prior mass of $\mathcal{B}_n \backslash \mathcal{B}_n^{*}$ and $\mathcal{H} \backslash \mathcal{H}_n$, we have that $$\mathbb{E}_{0}\left[\Pi\left(D | Y^{(n)}\right) \mathds{1}_{\mathcal{A}_n} (1 - \varphi_n)\right] \underset{n \rightarrow \infty}{\longrightarrow} 0,$$ and finally $\mathbb{E}_{0}\left[\Pi\left(D | Y^{(n)}\right)\right]\underset{n \rightarrow \infty}{\longrightarrow} 0$, what was wanted to be demonstrated.

Thus, to complete the proof, we need to demonstrate the existence of such a test $\varphi_n$ satisfying \eqref{test_phi} on an appropriate sieve $\mathcal{B}_n^{*} \times \mathcal{H}_n$ such that the prior mass of $\mathcal{B}_n \backslash \mathcal{B}_n^{*}$ and $\mathcal{H} \backslash \mathcal{H}_n$ have an exponential decrease.

\paragraph{Construction of the test $\varphi_n$:} To this end, we want to apply Lemma D.3 of \citet{ghosal2017fundamentals}, which directly allows to construct the test $\varphi_n$ with appropriate control of error probabilities as described in \eqref{test_phi} to test the true value against the whole of the alternative intersected with the sieve. To apply this lemma, we need to construct local tests with exponentially small errors to compare the true value with a subset of the alternative, centered at any $(\beta_1, \Gamma_1) \in \mathcal{B} \times \mathcal{H}$ which is adequately distant from the true value with respect to the average Rényi divergence. The other condition to apply this Lemma is that the minimum number $N_n^{*}$ of these small subsets of the alternative needed to cover a sieve $\mathcal{B}_n^{*} \times \mathcal{H}_n$ is appropriately controlled in terms of $\epsilon_n$.

First, the following lemma constructs an appropriate local test by employing the likelihood ratio to compare the true value with a subset of the alternative and by controlling the second order moment of the likelihood ratios in these small pieces of the alternative. For $(\beta_1, \Gamma_1) \in \mathcal{B} \times \mathcal{H}$, we denote by $p_1$ the associated density, and $\mathbb{E}_{1}$ and $\mathbb{P}_{1}$ the expectation and probability under $p_1$.

\begin{lemma}
\label{Lem_test}
    For a given positive sequence $(\gamma_n)$, $(\beta_1, \Gamma_1) \in \mathcal{B} \times \mathcal{H}$ such that $R_n(p_0,p_1) \geq \epsilon_n^2$, where $\epsilon_n~=~\sqrt{s_0 \log(p)/n}$, define
    \begin{align*}
        \mathcal{F}_{1,n} = \bigg\{ (\beta, \Gamma) \in \mathcal{B} \times \mathcal{H} : \: \frac{1}{n}\sum_{i=1}^n \|f_i(X_i\beta)-f_i(X_i\beta_1)\|_2^2 \leq \dfrac{\epsilon_n^2}{16 \gamma_n},d_n(\Gamma, \Gamma_1) \leq \dfrac{\epsilon_n^2}{2 M_{\text{obs}} \gamma_n}, \: \max_{1 \leq i \leq n} \| \Delta_{\Gamma,i}^{-1} \|_{sp} \leq \gamma_n \bigg\}.
    \end{align*}
    Grant Assumptions~\ref{hyp_ni_sup_q}-\ref{hyp_Zi_sp} and \ref{hyp_beta0}, then there exists a test $\overline{\varphi}_n$ such that $$\mathbb{E}_0[\overline{\varphi}_n] \leq e^{-n \epsilon_n^2} \text{, \: and } \: \sup_{(\beta, \Gamma) \in \mathcal{F}_{1,n}} \mathbb{E}_{\beta, \Gamma}[1-\overline{\varphi}_n] \leq e^{-n \epsilon_n^2/16}.$$
\end{lemma}
This lemma is demonstrated in Appendix \ref{appA}.

Now, we still have to construct an appropriate sieve $\mathcal{B}_n^{*} \times \mathcal{H}_n$ such that the prior mass of $\mathcal{B}_n \backslash \mathcal{B}_n^{*}$ and $\mathcal{H} \backslash \mathcal{H}_n$ have an exponential decrease, and the minimum number $N_n^{*}$ of the small subsets of the alternative needed to cover the sieve satisfies $\log(N_n^*) \lesssim n\epsilon_n^2$.

Define the sieve as follows:
\begin{align*}
    \mathcal{B}_n^{*}&=\left\{ \beta \in \mathcal{B} | s_{\beta} \leq C_1 s_0, \|\beta\|_{\infty} \leq \dfrac{p^{L_2+2}}{K' \|X\|_{*}} \right\}, \\
    \mathcal{H}_n&=\left\{ \Gamma \in \mathcal{H} | n^{-M} \leq \rho_{min}(\Gamma) \leq \rho_{max}(\Gamma) \leq e^{M n \epsilon_n^2} \right\},
\end{align*}
for a constant $M$, and define $\mathcal{F}_{1,n}$ as in Lemma~\ref{Lem_test} with $\gamma_n = n^M/\underline{\rho_Z}^2$. Remark that, with this choice of $\gamma_n$, the last condition $\max_{1 \leq i \leq n} \| \Delta_{\Gamma,i}^{-1} \|_{sp} \leq \gamma_n$ is always satisfy in the sieve. Indeed, $\| \Delta_{\Gamma,i}^{-1} \|_{sp} = \rho_{max}(\Delta_{\Gamma,i}^{-1}) = \rho_{min}^{-1}(\Delta_{\Gamma,i})$. But by Assumption~\ref{hyp_rg_Zi} and since $\Gamma \in \mathcal{H}_n$, $$\rho_{min}(\Delta_{\Gamma,i}) \geq \sigma^2 + \rho_{min}(Z_i \Gamma Z_i^\top) \geq \rho_{min}(\Gamma)\underline{\rho_Z}^2 \geq n^{-M} \underline{\rho_Z}^2.$$ So finally, $\max_{1 \leq i \leq n} \| \Delta_{\Gamma,i}^{-1} \|_{sp} \leq \gamma_n$ for $\Gamma \in \mathcal{H}_n$. 

First, we show that $\mathcal{B}_n \backslash \mathcal{B}_n^{*}$ and $\mathcal{H} \backslash \mathcal{H}_n$ have an exponential decrease.
Using Assumption~\ref{hyp_pi_p}, we obtain that:
\begin{align*}
    \Pi(\mathcal{B}_n \backslash \mathcal{B}_n^{*}) &= \Pi \left( \left\{ \beta \in \mathcal{B} | s_{\beta} \leq C_1 s_0, \|\beta\|_{\infty} > \dfrac{p^{L_2+2}}{K' \|X\|_{*}} \right\} \right) \\
    &= \sum_{S : s \leq C_1 s_0} \dfrac{\pi_p(s)}{\binom{\db}{s}} \int_{\left\{\beta_S : \|\beta_S\|_{\infty} > \dfrac{p^{L_2+2}}{K' \|X\|_{*}} \right\}} g_S(\beta_S)d\beta_S \\
    &\leq \sum_{S : s \leq C_1 s_0} \dfrac{(A_2p^{-A_4})^s}{\binom{\db}{s}} \int_{\left\{\beta_S : \|\beta_S\|_{\infty} > \dfrac{p^{L_2+2}}{K' \|X\|_{*}} \right\}} g_S(\beta_S)d\beta_S \\
    &\leq \sum_{S : s \leq C_1 s_0} \dfrac{(A_2p^{-A_4})^s}{\binom{\db}{s}} \sum_{\ell\in S} \int_{\left\{|\beta_{\ell}| > \dfrac{p^{L_2+2}}{K' \|X\|_{*}} \right\}} \dfrac{\lambda}{2}e^{-\lambda|\beta_{\ell}|}d\beta_{\ell}
\end{align*}
Then, by using the tail probability of the Laplace distribution $$\int_{|x|>t} \frac{\lambda}{2}e^{-\lambda|x|}dx=e^{-\lambda t},$$ for every $t>0$, and since there is $\binom{\db}{s}$ support $S$ of size $s$, we obtain:
\begin{align*}
    \Pi(\mathcal{B}_n \backslash \mathcal{B}_n^{*}) &\leq \sum_{S : s \leq C_1 s_0} \dfrac{(A_2p^{-A_4})^s}{\binom{\db}{s}} s e^{-\lambda \dfrac{p^{L_2+2}}{K' \|X\|_{*}}} \\
    &\leq \sum_{s=1}^{C_1 s_0} s (A_2p^{-A_4})^s e^{-\lambda \dfrac{p^{L_2+2}}{K' \|X\|_{*}}} \\
    &\leq C_1 s_0 e^{-\lambda \dfrac{p^{L_2+2}}{K' \|X\|_{*}}} \sum_{s=1}^{C_1 s_0} (A_2p^{-A_4})^s \\
    &\lesssim s_0 e^{-\lambda \dfrac{p^{L_2+2}}{K' \|X\|_{*}}} \lesssim s_0 e^{-\dfrac{\lambda}{L_1} p^2}
\end{align*}
Thus, $\Pi(\mathcal{B}_n \backslash \mathcal{B}_n^{*})e^{C n \epsilon_n^2} \underset{n \rightarrow \infty}{\longrightarrow} 0$ for every $C>0$ since $n \epsilon_n^2 = s_0\log(p) = o(p^2)$.

Now, 
\begin{align*}
    \Pi(\mathcal{H} \backslash \mathcal{H}_n) &=\Pi\left( \left\{ \Gamma \in \mathcal{H} | \rho_{min}(\Gamma) <  n^{-M}  \text{ or } \rho_{max}(\Gamma) > e^{M n \epsilon_n^2} \right\} \right) \\
    &\leq \Pi\left(\left\{ \Gamma \in \mathcal{H} | \rho_{min}(\Gamma) <  n^{-M}\right\}\right) + \Pi\left(\left\{ \Gamma \in \mathcal{H} | \rho_{max}(\Gamma) > e^{M n\epsilon_n^2} \right\}\right) \\
    &= \Pi\left(\left\{ \Gamma \in \mathcal{H} | \rho_{max}(\Gamma^{-1}) \geq  n^{M}\right\}\right) + \Pi\left(\left\{ \Gamma \in \mathcal{H} | \rho_{min}(\Gamma^{-1}) \leq e^{-M n\epsilon_n^2} \right\}\right) \\
    &\leq b_1 e^{-b_2 n^{b_3 M}} \times b_4e^{-b_5 M n \epsilon_n^2}
\end{align*}
for some constants $b_1, b_2, b_3, b_4, b_5>0$ by Lemma 9.16 of \citet{ghosal2017fundamentals} since $\Gamma^{-1}~\sim~\mathcal{W}_{\dg}(d, \Sigma^{-1})$. So, $\Pi(\mathcal{H} \backslash \mathcal{H}_n)e^{C n \epsilon_n^2} \underset{n \rightarrow \infty}{\longrightarrow} 0$ for every $C>0$, for $M$ large enough.

Finally, we have to prove that the minimum number $N_n^{*}$ of the small subsets of the alternative of the form $\mathcal{F}_{1,n}$ needed to cover the sieve satisfies $\log(N_n^*) \lesssim n\epsilon_n^2$. First, note that for every $\beta, \beta' \in \mathcal{B}$, by Assumption~\ref{hyp_lip}, the inequality $\|X \theta\|_2 \leq \|X\|_{*}\|\theta\|_1$, we have:

\begin{align*}
    \frac{1}{n}\sum_{i=1}^n \|f_i(X_i\beta)-f_i(X_i\beta')\|_2^2 &\leq \frac{1}{n}\sum_{i=1}^n K'^2\|X_i(\beta-\beta')\|_2^2 = \dfrac{K'^2}{n}\|X(\beta-\beta')\|_2^2 \\
    &\leq \dfrac{K'^2}{n}\|X\|_{*}^2 \|\beta-\beta'\|_1^2 \\
    &\leq \dfrac{(\db)^2 K'^2}{n}\|X\|_{*}^2 \|\beta-\beta'\|_{\infty}^2
\end{align*}
Thus, we define 
\begin{align*}
    \mathcal{F}_{1,n}' = \bigg\{ (\beta, \Gamma) \in \mathcal{B} \times \mathcal{H} : \: \dfrac{(\db)^2 K'^2}{n}\|X\|_{*}^2 \|\beta-\beta_1\|_{\infty}^2 + d_n^2(\Gamma, \Gamma_1) \leq \dfrac{1}{16 M_{\text{obs}}^2 \gamma_n^2 n^3}, \max_{1 \leq i \leq n} \| \Delta_{\Gamma,i}^{-1} \|_{sp} \leq \gamma_n \bigg\},
\end{align*}
with the same $(\beta_1, \Gamma_1)$ used in $\mathcal{F}_{1,n}$, and therefore we have $\mathcal{F}_{1,n}' \subset \mathcal{F}_{1,n}$. Thus, since $N_n^{*}$ is the minimum number of the small subsets of the alternative of the form $\mathcal{F}_{1,n}$ needed to cover the sieve $\mathcal{B}_n^{*} \times \mathcal{H}_n$, with $\mathcal{F}_{1,n} \supset \mathcal{F}_{1,n}'$, we have that $N_n^{*}$ is bounded above by the minimum number of the small subsets of the alternative of the form $\mathcal{F}_{1,n}'$ needed to cover the sieve $\mathcal{B}_n^{*} \times \mathcal{H}_n$. This last minimum number is denoted by $N_n'$. In the following, for a pseudo-metric space $(\mathcal{F},d)$, let $N(\epsilon,\mathcal{F},d)$ denote the minimal number of $\epsilon$-balls that cover $\mathcal{F}$.

Now, note that if $(\beta, \Gamma) \in \mathcal{B}_n^{*} \times \mathcal{H}_n$ with $\|\beta -\beta_1\|_{\infty} \leq \dfrac{1}{6n \db K' M_{\text{obs}} \gamma_n \|X\|_*}$ and $d_n(\Gamma,\Gamma_1) \leq \dfrac{1}{6 n^{3/2} M_{\text{obs}} \gamma_n}$, then $(\beta, \Gamma) \in \mathcal{F}_{1,n}'$ (by using that the last condition is satisfy for all $\Gamma \in \mathcal{H}_n$). Thus,
\begin{equation*}
    N_n' \leq N\left(\dfrac{1}{6n \db K' M_{\text{obs}} \gamma_n \|X\|_*}, \mathcal{B}_n^{*}, \| \cdot \|_{\infty} \right) \times N\left(\dfrac{1}{6 n^{3/2} M_{\text{obs}} \gamma_n}, \mathcal{H}_n, d_n \right).
\end{equation*}
Then, \small{\begin{equation}
    \label{eq_borne_Nn}
    \log(N_n^{*}) \leq \log N\left(\dfrac{1}{6n \db K' M_{\text{obs}} \gamma_n \|X\|_*}, \mathcal{B}_n^{*}, \| \cdot \|_{\infty} \right) +  \log N\left(\dfrac{1}{6 n^{3/2} M_{\text{obs}} \gamma_n}, \mathcal{H}_n, d_n \right).
\end{equation}}
Recall that $\mathcal{B}_n^{*}=\left\{ \beta \in \mathcal{B} | s_{\beta} \leq C_1 s_0, \|\beta\|_{\infty} \leq \dfrac{p^{L_2+2}}{K' \|X\|_{*}} \right\}$, to cover $\mathcal{B}_n^{*}$, we have to choose at most $\lfloor C_1 s_0 \rfloor$ non-zero $\beta$ coordinates and we need to recover a ball in $\mathbb{R}^{\lfloor C_1 s_0 \rfloor}$ with radius $\dfrac{p^{L_2+2}}{K' \|X\|_{*}}$, with balls of radius $\dfrac{1}{6n \db K' M_{\text{obs}} \gamma_n \|X\|_*}$. Therefore, 

\begin{align*}
    N\bigg(\dfrac{1}{6n \db K' M_{\text{obs}} \gamma_n \|X\|_*}, \mathcal{B}_n^{*}, \| \cdot \|_{\infty} \bigg) &\leq \binom{\db}{\lfloor C_1 s_0 \rfloor} \left( 6 p^{L_2+3} q n M_{\text{obs}} \gamma_n  \right)^{\lfloor C_1 s_0 \rfloor} \\
    &\lesssim \left( 6 p^{L_2+4} q^2 n M_{\text{obs}} \gamma_n  \right)^{\lfloor C_1 s_0 \rfloor} \text{ as } \binom{\db}{\lfloor C_1 s_0 \rfloor} (\lfloor C_1 s_0 \rfloor)! \leq (\db)^{\lfloor C_1 s_0 \rfloor}
\end{align*}
So, the first term in the right side of equation \eqref{eq_borne_Nn} is bounded by:
\begin{equation*}
    \log N\left(\dfrac{1}{6n \db K' M_{\text{obs}} \gamma_n \|X\|_*}, \mathcal{B}_n^{*}, \| \cdot \|_{\infty} \right) \lesssim s_0 \log(p) = n\epsilon_n^2
\end{equation*}
as $\log(n) \lesssim \log(p)$.

Similarly, for the second term in the right side of equation \eqref{eq_borne_Nn}, note that $\mathcal{H}_n \subset \left\{ \Gamma \in \mathcal{H} : \|\Gamma\|_F \leq \sqrt{\dg} e^{M n \epsilon_n^2} \right\}$, and therefore by assumption \ref{hyp_Zi_sp}: 
\begin{align*}
    \log N\bigg(\dfrac{1}{6 n^{3/2} M_{\text{obs}} \gamma_n}, \mathcal{H}_n, d_n \bigg) &\leq \log N\left(\dfrac{1}{6 n^{3/2} M_{\text{obs}} \gamma_n}, \left\{ \Gamma \in \mathcal{H} : \|\Gamma\|_F \leq \sqrt{\dg} e^{M n \epsilon_n^2} \right\}, d_n \right) \\
    &\leq \log N\left(\dfrac{1}{6 n^{3/2} M_{\text{obs}} \gamma_n \overline{\rho_Z}^2}, \left\{ \Gamma \in \mathcal{H} : \|\Gamma\|_F \leq \sqrt{\dg} e^{M n \epsilon_n^2} \right\}, \| \cdot \|_F \right) \\
    &\lesssim \dg(\dg+1) \log(\sqrt{\dg} e^{M n \epsilon_n^2} n^{3/2} M_{\text{obs}} \gamma_n) \lesssim n\epsilon_n^2.
\end{align*}

Finally, $\log(N_n^*) \lesssim n \epsilon_n^2$. Thus, Lemma D.3 of \citet{ghosal2017fundamentals} can be applied and gives that for every $\epsilon>\epsilon_n$, there exists a test $\varphi_n$ satisfying 
\begin{equation*}
    \mathbb{E}_0[\varphi_n] \leq 2 e^{B_1 n \epsilon_n^2-n\epsilon^2} \text{ and } \sup_{(\beta, \Gamma) \in \mathcal{B}_n^{*} \times \mathcal{H}_n : R_n^{*}(\beta, \Gamma) > \epsilon^2 } \mathbb{E}_{(\beta, \Gamma)}[1-\varphi_n] \leq e^{- n \epsilon^2/16}
\end{equation*}
for some constant $B_1>0$. Then, choosing $\epsilon=C_2 \epsilon_n$ for $C_2$ large enough, we obtain that the test $\varphi_n$ satisfies \eqref{test_phi}, which concludes the proof, as demonstrated above. 

\end{proof}

\subsection{Proof of Theorem \ref{Th_recovery}}
\label{Proof_Th3}

\begin{proof}[Proof of Theorem \ref{Th_recovery}]

The contraction rate of the posterior distribution with respect to the average Rényi divergence $R_n^{*}(\beta, \Gamma)=R_n(p_{\beta, \Gamma},p_0)$ is provided by Theorem \ref{Th_Renyi}. Denote $\epsilon_n = \sqrt{s_0 \log(p)/n}$ this rate. We have that, for all $\beta_0 \in \mathcal{B}_0, \Gamma_0 \in \mathcal{H}_0$,
$$\mathbb{E}_{0}\left[\Pi\left(\mathcal{R}| Y^{(n)}\right)\right] \underset{n \rightarrow \infty}{\longrightarrow} 1.$$
where $\mathcal{R}= \left\{(\beta,\Gamma) \in \mathcal{B} \times \mathcal{H} : R_n^{*}(\beta, \Gamma) \leq C_2 \epsilon_n^2 \right\}$. However, since for all $(\beta,\Gamma) \in\mathcal{B} \times\mathcal{H}$, $p_{\beta, \Gamma}= \prod_{i=1}^n p_{\beta, \Gamma,i}$, with $p_{\beta, \Gamma,i}= \mathcal{N}_{m_i}(f_i(X_i\beta),\Delta_{\Gamma,i})$ in the model~\eqref{NLMM}, the average Rényi divergence is equal to:

\small{
\begin{align*}
    R_n^{*}(\beta, \Gamma) &= R_n(p_{\beta, \Gamma},p_0)=-\dfrac{1}{n}\sum_{i=1}^n \log \left( \int \sqrt{p_{\beta, \Gamma,i}(y_i) p_{0,i}(y_i)}dy_i \right) \\
    &= -\dfrac{1}{n}\sum_{i=1}^n \left[\log \left( 1-g^2(\Delta_{\Gamma,i}, \Delta_{\Gamma_0,i}) \right) - \dfrac{1}{4} \|(\Delta_{\Gamma,i}+ \Delta_{\Gamma_0,i})^{-1/2}(f_i(X_i\beta)-f_i(X_i\beta_0))\|_2^2 \right]
\end{align*}}
where we used the Sherman-Morrison-Woodbury formula, with 
\begin{equation*}
    g^2(\Delta_{\Gamma,i}, \Delta_{\Gamma_0,i})=1-\dfrac{\det(\Delta_{\Gamma,i})^{1/4} \det(\Delta_{\Gamma_0,i})^{1/4}}{\det((\Delta_{\Gamma,i}+\Delta_{\Gamma_0,i})/2)^{1/2}}.
\end{equation*}
Remark that for all $1\leq i \leq n$, $g^2(\Delta_{\Gamma,i}, \Delta_{\Gamma_0,i})\geq0$ since, with $\Delta_{\Gamma,i}^{*} ~ = ~\Delta_{\Gamma_0,i}^{-1/2} \Delta_{\Gamma,i} \Delta_{\Gamma_0,i}^{-1/2}$
\begin{align*}
    \dfrac{\det(\Delta_{\Gamma,i})^{1/4} \det(\Delta_{\Gamma_0,i})^{1/4}}{\det((\Delta_{\Gamma,i}+\Delta_{\Gamma_0,i})/2)^{1/2}}&=\left( \dfrac{1}{2^{m_i}}\det(\Delta_{\Gamma,i}^{*^{1/2}} + \Delta^{*^{-1/2}}_{\Gamma,i} )\right)^{-1/2} \\
    &= \left( \prod_{k=1}^{m_i} \dfrac{1}{2}(d_k^{1/2}+d_k^{-1/2})\right)^{-1/2} \leq 1
\end{align*}
where $d_k$ are the eigenvalues of $\Delta_{\Gamma,i}^{*}$, and using that $\forall x \leq 0, x+x^{-1} \geq 2$.

Thus, by Theorem \ref{Th_Renyi}, this implies that, for $(\beta,\Gamma) \in \mathcal{R}$:
\begin{align*}
    \epsilon_n^2 &\gtrsim -\dfrac{1}{n}\sum_{i=1}^n \log \left( 1-g^2(\Delta_{\Gamma,i}, \Delta_{\Gamma_0,i}) \right) \\
    &\gtrsim \dfrac{1}{n}\sum_{i=1}^n g^2(\Delta_{\Gamma,i}, \Delta_{\Gamma_0,i}) 
\end{align*}
since $\log(1-x)\leq -x$ for all $x\geq 0$. Now, by Lemma $10$ of \citet{jeong_unified_2021}, for each $i \in \{1, \dots, n\}$, we obtain that $g^2(\Delta_{\Gamma,i}, \Delta_{\Gamma_0,i}) \gtrsim \|\Delta_{\Gamma,i} -\Delta_{\Gamma_0,i} \|_F^2 $ if $g^2(\Delta_{\Gamma,i}, \Delta_{\Gamma_0,i})$ is small enough. Thus, by defining $I_{n,\delta}=\{ 1 \leq i \leq n : g^2(\Delta_{\Gamma,i}, \Delta_{\Gamma_0,i}) \geq \delta_n\}$, for $\delta_n=o(1)$ such that $\epsilon_n^2=o(\delta_n)$, we have:

\begin{align*}
    \epsilon_n^2 &\gtrsim \dfrac{1}{n}\sum_{i\notin I_{n,\delta}} \|\Delta_{\Gamma,i} -\Delta_{\Gamma_0,i} \|_F^2 + \dfrac{1}{n}\sum_{i \in I_{n,\delta}} g^2(\Delta_{\Gamma,i}, \Delta_{\Gamma_0,i}) \\
    &\gtrsim \dfrac{1}{n}\sum_{i\notin I_{n,\delta}} \|\Delta_{\Gamma,i} -\Delta_{\Gamma_0,i} \|_F^2 = \dfrac{1}{n}\sum_{i=1}^n \|\Delta_{\Gamma,i} -\Delta_{\Gamma_0,i} \|_F^2 - \dfrac{1}{n}\sum_{i\in I_{n,\delta}} \|\Delta_{\Gamma,i} -\Delta_{\Gamma_0,i} \|_F^2 \\
    &\geq M_1 d^2_n(\Gamma,\Gamma_0) - M_1 \dfrac{|I_{n,\delta}|}{n} \max_{1\leq i \leq n} \|\Delta_{\Gamma,i} -\Delta_{\Gamma_0,i} \|_F^2 \\
    & \geq M_1 d^2_n(\Gamma,\Gamma_0) - M_2 \dfrac{\epsilon_n^2}{\delta_n} \max_{1\leq i \leq n} \|\Delta_{\Gamma,i} -\Delta_{\Gamma_0,i} \|_F^2 \\
    & \geq (M_1- M_3 \dfrac{\epsilon_n^2}{\delta_n}) d^2_n(\Gamma,\Gamma_0)
\end{align*}
where we used that $\dfrac{|I_{n,\delta}|}{n} \lesssim \dfrac{\epsilon_n^2}{\delta_n}$ since
$$\epsilon_n^2 ~\gtrsim~ \dfrac{1}{n}\sum_{i\in I_{n,\delta}} g^2(\Delta_{\Gamma,i}, \Delta_{\Gamma_0,i})~ \gtrsim~ \dfrac{|I_{n,\delta}|}{n} \times \delta_n,$$
and thanks to Lemma \ref{Lem_norm_Gamma} of Appendix \ref{appB} for the last inequality. Then, since $\epsilon_n^2=o(\delta_n)$, $M_1-M_3\epsilon_n^2/\delta_n$ is bounded away from $0$, the last inequation implies that $\epsilon_n \gtrsim d_n(\Gamma, \Gamma_0)$, which proves the first assertion of Theorem~\ref{Th_recovery}. Now, thanks to Lemma \ref{Lem_norm_Gamma} of Appendix \ref{appB}, we have also that $\epsilon_n \gtrsim \|\Gamma - \Gamma_0\|_F$, which proves the second assertion of Theorem~\ref{Th_recovery}.

Also, by Theorem \ref{Th_Renyi}, for $(\beta,\Gamma) \in \mathcal{R}$, since $g^2(\Delta_{\Gamma,i}, \Delta_{\Gamma_0,i}) \geq 0$, we have that:
\begin{align*}
    \epsilon_n^2 &\gtrsim \dfrac{M_4}{4n}\sum_{i=1}^n \|(\Delta_{\Gamma,i}+\Delta_{\Gamma_0,i})^{-1/2}(f_i(X_i\beta)-f_i(X_i\beta_0))\|_2^2 \\
    &\geq \dfrac{M_4}{4n}\sum_{i=1}^n \rho_{min}((\Delta_{\Gamma,i}+\Delta_{\Gamma_0,i})^{-1} )\|f_i(X_i\beta)-f_i(X_i\beta_0)\|_2^2 \\
\end{align*}
using that for $A$ symmetric matrix, $\|Ax\|_2^2 \geq \rho_{min}(A^2) \|x\|_2^2$ for all vector $x$.

Now, since $\rho_{min}((\Delta_{\Gamma,i}+\Delta_{\Gamma_0,i})^{-1})=\rho_{max}^{-1}(\Delta_{\Gamma,i}+\Delta_{\Gamma_0,i})$, we want to upper bound $\rho_{max}(\Delta_{\Gamma,i}+\Delta_{\Gamma_0,i})$ uniformly across $i \in \{1, \dots, n\}$. Thus, by using Weyl's inequality:

\begin{equation*}
    \rho_{max}(\Delta_{\Gamma,i}+\Delta_{\Gamma_0,i}) \leq \rho_{max}(\Delta_{\Gamma,i}-\Delta_{\Gamma_0,i}) + 2\rho_{max}(\Delta_{\Gamma_0,i}) \leq \|\Delta_{\Gamma,i}-\Delta_{\Gamma_0,i}\|_F + 2 \overline{\rho_{\Delta}}
\end{equation*}
by Lemma \ref{vp_Delta} of Appendix \ref{appB} where $\overline{\rho_{\Delta}}$ denotes the uniform upper bound of the eigenvalues of $(\Delta_{\Gamma_0,i})_i$. Thus, by Lemma \ref{Lem_norm_Gamma} of Appendix \ref{appB},
\begin{align*}
    \max_{1 \leq i \leq n} \rho_{max}(\Delta_{\Gamma,i}+\Delta_{\Gamma_0,i}) &\leq \max_{1 \leq i \leq n} \|\Delta_{\Gamma,i}-\Delta_{\Gamma_0,i}\|_F + 2 \overline{\rho_{\Delta}} \\
    &\leq d_n(\Gamma,\Gamma_0) + 2 \overline{\rho_{\Delta}} \leq C_3 \epsilon_n + 2 \overline{\rho_{\Delta}},
\end{align*}
by the first assertion of Theorem \ref{Th_recovery}.
Finally, we obtain that:
\begin{align*}
    \epsilon_n^2 &\geq \dfrac{M_4}{4n (C_3 \epsilon_n + 2 \overline{\rho_{\Delta}})}\sum_{i=1}^n\|f_i(X_i\beta)-f_i(X_i\beta_0)\|_2^2,
\end{align*}
and since $C_3 \epsilon_n + 2 \overline{\rho_{\Delta}}  \underset{n \rightarrow \infty}{\longrightarrow} 2 \overline{\rho_{\Delta}}$, we finally obtain that for $n$ large enough $$\epsilon_n ~\gtrsim ~\sqrt{\dfrac{1}{n}\sum_{i=1}^n\|f_i(X_i\beta)-f_i(X_i\beta_0)\|_2^2},$$ which gives the last assertion of Theorem \ref{Th_recovery}.

\end{proof}

\subsection{Proof of Theorem \ref{Th_beta}}
\label{Proof_Th4}

\begin{proof}[Proof of Theorem \ref{Th_beta}]
Let us consider $\beta_0 \in \mathcal{B}_0$ and $\Gamma_0 \in \mathcal{H}_0$.
The contraction rate of the posterior distribution for the prediction term is provided by Theorem \ref{Th_recovery} and we have in particular that: for all $\beta_0 \in \mathcal{B}_0, \Gamma_0 \in \mathcal{H}_0$,
$$\mathbb{E}_{0}\left[\Pi\left( \mathcal{P}_n \bigg| Y^{(n)}\right)\right] \underset{n \rightarrow \infty}{\longrightarrow} 1,$$
where $\mathcal{P}_n=\left\{\beta : \frac{1}{n}\sum_{i=1}^n \|f_i(X_i\beta)-f_i(X_i\beta_0)\|_2^2 \lesssim\epsilon_n^2\right\}$ and $\epsilon_n=\sqrt{s_0\log(p)/n}$. 

Thus, for $\beta \in \mathcal{P}_n$, and for $n$ large enough we can apply Assumption \ref{hyp_identifiability} for $\delta=\epsilon_n^2$ and we obtain:
$\|X(\beta-\beta_0)\|_2^2 \lesssim  n\epsilon_n^{2\eta}$ which gives the first assertion of Theorem \ref{Th_beta}.

Remark that for $\beta$ such as $s_{\beta} \leq C_1 s_0$, we have $s_{\beta-\beta_0} \leq s_{\beta} + s_{0} \leq (C_1+1)s_0$, thus by Theorem \ref{Th_dim} and definition of the uniform compatibility number $\phi_1$ and the smallest scaled singular value $\phi_2$, we obtain that:
\begin{align*}
    \epsilon_n^{2\eta} &\gtrsim \dfrac{\|X\|_{*}^2 \phi_1^2((C_1+1)s_0)}{s_0 n} \|\beta-\beta_0\|_1^2,  \\
    \text{and }\epsilon_n^{2\eta} &\gtrsim \dfrac{\|X\|_{*}^2 \phi_2^2((C_1+1)s_0)}{n} \|\beta-\beta_0\|_2^2,
\end{align*}
which proves the last two assertions of the theorem.
\end{proof}

\section*{Acknowledgements}
This work was funded by the Stat4Plant project ANR-20-CE45-0012. The authors thank Professor Ismaël Castillo (Sorbonne Université) for his helpful comments and recommendations about this work. The authors would like to thank the anonymous referees, an Associate Editor and the Editor for their constructive comments that improved the quality of this paper.

\section*{Funding}
This work was funded by the Stat4Plant project ANR-20-CE45-0012.

\bibliography{bibliography}
\bibliographystyle{apalike}

\begin{appendices}

\setcounter{lemma}{0}
\renewcommand{\thelemma}{B\arabic{lemma}}
\setcounter{theorem}{0}
\renewcommand{\thetheorem}{B\arabic{theorem}}

\section{Proofs of technical lemmas}\label{appA}
\subsection{Proof of Lemma \ref{Lem_denom}}
\label{Proof_Lem1}

First, we define a Kullback-Leibler neighbourhood around $p_{0,i}$
\small{
\[ \mathcal{D}_n = \left\{(\beta,\Gamma)\in \mathcal{B} \times \mathcal{H} \: \bigg| \: \sum_{i=1}^n K(p_{0,i},p_{\beta,\Gamma,i}) \leq c_1 \log(n), \: \sum_{i=1}^n V(p_{0,i},p_{\beta,\Gamma,i}) \leq c_1 \log(n) \right\}\]}
for a constant $c_1$ large enough.
Then,  by Lemma 10 of \citet{ghosal_convergence_2007}, we have that, for every $C>0$, 
\begin{equation}\label{ineq:int_Dn}
\mathbb{P}_0 \left(\int_{\mathcal{D}_n} \Lambda_n(\beta,\Gamma) d\Pi(\beta,\Gamma) \leq e^{-(1+C)c_1 \log(n)} \Pi(\mathcal{D}_n) \right)\leq \dfrac{1}{C^2 c_1 \log(n)}.
\end{equation}

\raggedright
The proof of the lemma consists in showing that there exists $M>0$ such that $e^{-(1+C)c_1 \log(n)}\Pi(\mathcal{D}_n) \gtrsim \pi_p(s_0) e^{-M(s_0 \log(p) + \log(n))}$. Indeed by combining this result with Inequality~\eqref{ineq:int_Dn}, since $\int_{\mathcal{D}_n} \Lambda_n(\beta,\Gamma) d\Pi(\beta,\Gamma) \leq \int \Lambda_n(\beta,\Gamma) d\Pi(\beta,\Gamma)$, we have that
\begin{align*}
\mathbb{P}_0 \left(\int \Lambda_n(\beta,\Gamma) d\Pi(\beta,\Gamma) \geq   \pi_p(s_0) e^{-M(s_0 \log(p) + \log(n))} \right) &\geq \mathbb{P}_0 \left(\int_{\mathcal{D}_n} \Lambda_n(\beta,\Gamma) d\Pi(\beta,\Gamma) \geq   \pi_p(s_0) e^{-M(s_0 \log(p) + \log(n))} \right) \\
&\geq 1 - \mathbb{P}_0 \left(\int_{\mathcal{D}_n} \Lambda_n(\beta,\Gamma) d\Pi(\beta,\Gamma) \leq e^{-(1+C)c_1 \log(n)} \Pi(\mathcal{D}_n) \right) \\
&\geq 1 - \dfrac{1}{C^2 c_1 \log(n)} \underset{n \rightarrow \infty}{\longrightarrow} 1,
\end{align*}
that concludes the proof.
Thus, it remains to show that $e^{-(1+C)c_1 \log(n)}\Pi(\mathcal{D}_n) \geq \pi_p(s_0) e^{-M(s_0 \log(p) + \log(n))}$, or, more precisely, we need to exhibit a lower bound of $\Pi(\mathcal{D}_n)$.
    
In the non-linear marginal mixed-effects model, we have that $p_{\beta,\Gamma,i}=\mathcal{N}(f_i(X_i \beta), \Delta_{\Gamma,i})$, with $\Delta_{\Gamma,i}= Z_i \Gamma Z_i^\top + \sigma^2 I_{m_i}$. By Lemma~9 of \citet{jeong_unified_2021}, the Kullback-Leibler divergence and variation of the $i$-th individual are respectively expressed as: 
    \begin{align*}
    K(p_{0,i},p_{\beta,\Gamma,i})=& \dfrac{1}{2} \left[ \log \left( \dfrac{| \Delta_{\Gamma,i}|}{|\Delta_{\Gamma_0,i} |} \right) + \text{Tr}(\Delta_{\Gamma_0,i} \Delta_{\Gamma,i}^{-1}) -m_i \right] + \dfrac{1}{2}\left\| \Delta_{\Gamma,i}^{-1/2}\left( f_i(X_i\beta)-f_i(X_i\beta_0)\right)\right\|_2^2 , \\
    V(p_{0,i},p_{\beta,\Gamma,i})=& \dfrac{1}{2} \left[ \text{Tr}\left( \Delta_{\Gamma_0,i} \Delta_{\Gamma,i}^{-1} \Delta_{\Gamma_0,i} \Delta_{\Gamma,i}^{-1} \right) -2 \text{Tr}(\Delta_{\Gamma_0,i} \Delta_{\Gamma,i}^{-1}) +m_i \right] + \left\| \Delta_{\Gamma_0,i}^{1/2}\Delta_{\Gamma,i}^{-1}\left( f_i(X_i\beta)-f_i(X_i\beta_0)\right)\right\| _2^2.
    \end{align*}
Then, by denoting $\rho_{i,k}$, for $k=1, \ldots, m_i$, the eigenvalues of $ \Delta_{\Gamma_0,i}^{1/2} \Delta_{\Gamma,i}^{-1} \Delta_{\Gamma_0,i}^{1/2}$, we obtain that
    \begin{align*}
    K(p_{0,i},p_{\beta,\Gamma,i})&= \dfrac{1}{2} \left[ - \sum_{k=1}^{m_i} \log(\rho_{i,k}) - \sum_{k=1}^{m_i} (1 -\rho_{i,k}) + \| \Delta_{\Gamma,i}^{-1/2}\left( f_i(X_i\beta)-f_i(X_i\beta_0)\right)\|_2^2 \right] \\
    V(p_{0,i},p_{\beta,\Gamma,i})&= \dfrac{1}{2} \sum_{k=1}^{m_i} (1- \rho_{i,k})^2 + \| \Delta_{\Gamma_0,i}^{1/2}\Delta_{\Gamma,i}^{-1}\left( f_i(X_i\beta)-f_i(X_i\beta_0)\right)\|_2^2.
    \end{align*}
Our goal is to find a lower bound of $\Pi(\mathcal{D}_n)$, so we want to find an upper bound of $\sum_{i=1}^n K(p_{0,i},p_{\beta,\Gamma,i})$ and $\sum_{i=1}^n V(p_{0,i},p_{\beta,\Gamma,i})$. 
    
Let us first focus on the term $V(p_{0,i},p_{\beta,\Gamma,i})$. By Lemma~10 of \citet{jeong_unified_2021}, we obtain that: $$ \sum_{k=1}^{m_i} (1-\rho_{i,k}^{-1})^2 \leq \rho_{min}^{-2}(\Delta_{\Gamma_0,i}) \| \Delta_{\Gamma,i} - \Delta_{\Gamma_0,i}\|_F^2.$$
By using Weyl's inequality and Assumptions~\ref{hyp_Zi_sp} and \ref{hyp_gamma0}, Lemma~\ref{vp_Delta} shows that there exist $\underline{\rho_0}>0$ and $\overline{\rho_0}>0$ such that:
\begin{equation}\label{appli:lemmaA2}
\underline{\rho_0} \leq \min_i \rho_{min}(\Delta_{\Gamma_0,i}) \leq \max_i \rho_{max}(\Delta_{\Gamma_0,i}) \leq \overline{\rho_0}. 
\end{equation}
Thus $$ \max_i \sum_{k=1}^{m_i} (1-\rho_{i,k}^{-1})^2 \leq  \underline{\rho_0}^{-2} \max_i \| \Delta_{\Gamma,i} - \Delta_{\Gamma_0,i}\|_F^2,$$
and in particular, $\max_{i,k} (1-\rho_{i,k}^{-1})^2 \leq  \underline{\rho_0}^{-2} \max_i \| \Delta_{\Gamma,i} - \Delta_{\Gamma_0,i}\|_F^2$.
Thus, if $\max_i \| \Delta_{\Gamma,i} - \Delta_{\Gamma_0,i}\|_F^2 \rightarrow 0$ on $\mathcal{D}_n$, we will have that $\max_{i,k} |1-\rho_{i,k}^{-1}|\rightarrow 0$, that is each $\rho_{i,k}$ tends to $1$ and so:
\begin{equation}
\label{Eq_borne1_V}
\sum_{k=1}^{m_i} (1-\rho_{i,k})^2 \lesssim \sum_{k=1}^{m_i} (1-\rho_{i,k}^{-1})^2 \leq  \underline{\rho_0}^{-2} \| \Delta_{\Gamma,i} - \Delta_{\Gamma_0,i}\|_F^2 \lesssim \| \Delta_{\Gamma,i} - \Delta_{\Gamma_0,i}\|_F^2,
\end{equation}
where the first inequality is due to $|1-x^{-1}| \lesssim |1-x| \lesssim |1-x^{-1}|$ for $x \rightarrow 1$, which enables to bound the first term of $V(p_{0,i},p_{\beta,\Gamma,i})$.

Now, prove that $\max_i \| \Delta_{\Gamma,i} - \Delta_{\Gamma_0,i}\|_F^2$ tends to $0$ on $\mathcal{D}_n$. We introduce the set $I_{n,\delta}=\{1 \leq i \leq n \vert \sum_{k=1}^{m_i} (1- \rho_{i,k})^2 \geq \delta_n\}$, for $\delta_n=o(1)$ and $\log(n)/n=o(\delta_n)$.  We denote by $|I_{n,\delta}|$ its cardinal.  Then for $(\beta, \Gamma) \in \mathcal{D}_n$, since $\sum_{i=1}^n V(p_{0,i},p_{\beta,\Gamma,i}) \leq c_1 \log(n)$, we have that, on the one hand:
    \begin{equation*}
        \sum_{i=1}^n  \sum_{k=1}^{m_i} (1- \rho_{i,k})^2 \leq c_1 \log(n),
    \end{equation*}
and on the other hand,
    \begin{align*}
    \sum_{i=1}^n  \sum_{k=1}^{m_i} (1- \rho_{i,k})^2 &= \sum_{i \in I_{n,\delta}}  \sum_{k=1}^{m_i} (1- \rho_{i,k})^2 + \sum_{i \notin I_{n,\delta}}  \sum_{k=1}^{m_i} (1- \rho_{i,k})^2 \\
    &\gtrsim \delta_n |I_{n,\delta}| + \sum_{i \notin I_{n,\delta}}  \sum_{k=1}^{m_i} \left(1- \dfrac{1}{\rho_{i,k}}\right)^2
    \end{align*}
    since for $i \notin I_{n,\delta}$,  $\sum_{k=1}^{m_i} (1- \rho_{i,k})^2 < \delta_n$, so each $|1-\rho_{i,k}|$ is less than $\sqrt{\delta_n}$ with $\delta_n=o(1)$, and we have that $|1-x^{-1}| \lesssim |1-x| \lesssim |1-x^{-1}|$ for $x \rightarrow 1$, so $|1- \rho_{i,k} | \gtrsim |1-\rho_{i,k}^{-1}|$. 
        Then, by using Lemma~10 of \citet{jeong_unified_2021}, we obtain that
    
    \[\sum_{i=1}^n  \sum_{k=1}^{m_i} (1- \rho_{i,k})^2 \gtrsim \delta_n |I_{n,\delta}| + \sum_{i \notin I_{n,\delta}}  \dfrac{1}{\rho_{\text{max}}^2(\Delta_{\Gamma_0,i})} \| \Delta_{\Gamma,i}-\Delta_{\Gamma_0,i} \|_F^2 \]
    
    By \eqref{appli:lemmaA2}, we obtain that
    \begin{equation*}
        c_1 \log(n) \geq \sum_{i=1}^n  \sum_{k=1}^{m_i} (1- \rho_{i,k})^2 \gtrsim \delta_n |I_{n,\delta}| + \dfrac{1}{\overline{\rho_0}^2}\sum_{i \notin I_{n,\delta}} \| \Delta_{\Gamma,i}-\Delta_{\Gamma_0,i} \|_F^2 
    \end{equation*}
    that is equivalent to
    \begin{equation*}
        \dfrac{\log(n)}{n} \gtrsim \dfrac{1}{n}\sum_{i=1}^n  \sum_{k=1}^{m_i} (1- \rho_{i,k})^2 \gtrsim \delta_n \dfrac{|I_{n,\delta}|}{n} + \dfrac{1}{n\overline{\rho_0}^2}\sum_{i \notin I_{n,\delta}} \| \Delta_{\Gamma,i}-\Delta_{\Gamma_0,i} \|_F^2
    \end{equation*}

In particular, $\delta_n \dfrac{|I_{n,\delta}|}{n} \lesssim \dfrac{\log(n)}{n}$, which implies $|I_{n,\delta}| \lesssim \log(n)/\delta_n$. We have also that $\dfrac{\log(n)}{n}~\gtrsim~\dfrac{1}{n\overline{\rho_0}^2}\sum_{i \notin I_{n,\delta}} \| \Delta_{\Gamma,i}-\Delta_{\Gamma_0,i} \|_F^2$, that is 
\begin{equation}
    \label{Eq_int1}
    \dfrac{\log(n)}{n} \gtrsim \dfrac{1}{n}\sum_{i \notin I_{n,\delta}} \| \Delta_{\Gamma,i}-\Delta_{\Gamma_0,i} \|_F^2.
\end{equation}
Then,
    \begin{align*}
        \dfrac{1}{n}\sum_{i \notin I_{n,\delta}} \| \Delta_{\Gamma,i}-\Delta_{\Gamma_0,i} \|_F^2 &= \dfrac{1}{n}\sum_{i=1}^n \| \Delta_{\Gamma,i}-\Delta_{\Gamma_0,i} \|_F^2 - \dfrac{1}{n}\sum_{i \in I_{n,\delta}} \| \Delta_{\Gamma,i}-\Delta_{\Gamma_0,i} \|_F^2 \\
        &\geq d_n^2(\Gamma,\Gamma_0) - \dfrac{|I_{n,\delta}|}{n} \max_{1 \leq i \leq n}\| \Delta_{\Gamma,i}-\Delta_{\Gamma_0,i} \|_F^2 
    \end{align*}
    
    By Lemma~\ref{Lem_norm_Gamma}, with Assumptions~\ref{hyp_ni_sup_q}, \ref{hyp_rg_Zi} and \ref{hyp_Zi_sp}, for $\Gamma_1,\Gamma_2\in\mathcal{H}$, we have that
$$\max_i \|\Delta_{\Gamma_1,i}-\Delta_{\Gamma_2,i} \|_F^2 \lesssim  \|\Gamma_1-\Gamma_2 \|_F^2 \lesssim d_n^2(\Gamma_1,\Gamma_2).$$
    
    Thus, there exist some constants $M_1>0$ and $M_2>0$:
    \begin{align}
        \dfrac{1}{n} \sum_{i \notin I_{n,\delta}} \| \Delta_{\Gamma,i}-\Delta_{\Gamma_0,i} \|_F^2 &\geq M_1 \max_{1 \leq i \leq n}\| \Delta_{\Gamma,i}-\Delta_{\Gamma_0,i} \|_F^2 - \dfrac{|I_{n,\delta}|}{n} \max_{1 \leq i \leq n}\| \Delta_{\Gamma,i}-\Delta_{\Gamma_0,i} \|_F^2 \nonumber\\
        &\geq \left(M_1-\dfrac{|I_{n,\delta}|}{n}\right) \max_{1 \leq i \leq n}\| \Delta_{\Gamma,i}-\Delta_{\Gamma_0,i} \|_F^2 \nonumber\\
        &\geq \left(M_1-M_2\dfrac{\log(n)}{n\delta_n}\right) \max_{1 \leq i \leq n}\| \Delta_{\Gamma,i}-\Delta_{\Gamma_0,i} \|_F^2   \label{up_Delta}
    \end{align}
    since $|I_{n,\delta}| \lesssim \log(n)/\delta_n$. Finally, by combining \eqref{Eq_int1} and \eqref{up_Delta}, and since $\log(n)/n=o(\delta_n)$, we obtain that, on $\mathcal{D}_n$, $$\max_{1 \leq i \leq n}\| \Delta_{\Gamma,i}-\Delta_{\Gamma_0,i} \|_F^2 \lesssim \dfrac{\log(n)}{n} \underset{n \rightarrow \infty}{\longrightarrow} 0.$$
Hence, by using Equation~\eqref{Eq_borne1_V}, we obtain that:
\begin{align*}
    \dfrac{1}{n} \sum_{i=1}^n V(p_{0,i},p_{\beta,\Gamma,i}) &= \dfrac{1}{2n} \sum_{i=1}^n \sum_{k=1}^{m_i} (1- \rho_{i,k})^2 + \dfrac{1}{n} \sum_{i=1}^n \| \Delta_{\Gamma_0,i}^{1/2}\Delta_{\Gamma,i}^{-1}\left( f_i(X_i\beta)-f_i(X_i\beta_0)\right)\|_2^2 \\
    & \lesssim \dfrac{1}{n} \sum_{i=1}^n \| \Delta_{\Gamma,i} - \Delta_{\Gamma_0,i}\|_F^2 + \dfrac{1}{n} \sum_{i=1}^n \| \Delta_{\Gamma_0,i}^{1/2}\|_{sp}^2 \|\Delta_{\Gamma,i}^{-1}\|_{sp}^2 \|f_i(X_i\beta)-f_i(X_i\beta_0)\|_2^2
\end{align*}
since $\|Ax\|_2 \leq \|A\|_{sp} \|x\|_2$. Then, by Lemma \ref{vp_Delta}, we have that $\| \Delta_{\Gamma_0,i}^{1/2}\|_{sp}^2=\rho_{max}(\Delta_{\Gamma_0,i}) \lesssim 1$ for each $i \in \{1, \dots, n\}$. Moreover, on $\mathcal{D}_n$,  by Lemma \ref{Lem_rho_max},
\begin{align*}
    \|\Delta_{\Gamma,i}^{-1}\|_{sp} &= \rho_{max}(\Delta_{\Gamma,i}^{-1}) = \rho_{max}(\Delta_{\Gamma_0,i}^{-1/2} \Delta_{\Gamma_0,i}^{1/2} \Delta_{\Gamma,i}^{-1} \Delta_{\Gamma_0,i}^{1/2} \Delta_{\Gamma_0,i}^{-1/2}) \\
    &\leq \rho_{max}(\Delta_{\Gamma_0,i}^{1/2} \Delta_{\Gamma,i}^{-1} \Delta_{\Gamma_0,i}^{1/2}) \| \Delta_{\Gamma_0,i}^{-1/2}\|_{sp}^2 \\
    & \lesssim \rho_{max}(\Delta_{\Gamma_0,i}^{1/2} \Delta_{\Gamma,i}^{-1} \Delta_{\Gamma_0,i}^{1/2}) = \max_k \rho_{i,k} \lesssim 1
\end{align*}
since $\| \Delta_{\Gamma_0,i}^{-1/2}\|_{sp}^2 = \rho_{max}(\Delta_{\Gamma_0,i}^{-1})=\rho_{min}(\Delta_{\Gamma_0,i})^{-1} \leq \underline{\rho_0}^{-1}$ by Lemma \ref{vp_Delta}, and since each $\rho_{i,k}$ tends to $1$ on $\mathcal{D}_n$. Then, we have to control the term $\|f_i(X_i\beta)-f_i(X_i\beta_0)\|_2^2$ for a non-linear function $f$. Thus, using Assumption~\ref{hyp_lip}, that is each $f_i$ is $K'$-Lipschitz, we deduce that:
$$\dfrac{1}{n} \sum_{i=1}^n V(p_{0,i},p_{\beta,\Gamma,i}) \lesssim d_n^2(\Gamma, \Gamma_0) + \dfrac{K'^2}{n} \|X(\beta-\beta_0)\|_2^2.$$

Let us now focus on the term $K(p_{0,i},p_{\beta,\Gamma,i})$. We have shown that on $\mathcal{D}_n$ each $|1-\rho_{i,k}|$ is small for each $i$ and $k$, so: $\log(\rho_{i,k})~=~\log(1-(1-\rho_{i,k})) \sim  -(1-\rho_{i,k}) - \dfrac{(1-\rho_{i,k})^2}{2}$, and so $-\log(\rho_{i,k})-(1-\rho_{i,k})\sim \dfrac{(1-\rho_{i,k})^2}{2}$. Thus, by using Inequality~\eqref{Eq_borne1_V}:
\begin{align*}
    \dfrac{1}{n} K(p_{0,i},p_{\beta,\Gamma,i})&= \dfrac{1}{2n} \sum_{i=1}^n \left[ - \sum_{k=1}^{m_i} \log(\rho_{i,k}) - \sum_{k=1}^{m_i} (1 -\rho_{i,k}) + \| \Delta_{\Gamma,i}^{-1/2}\left( f_i(X_i\beta)-f_i(X_i\beta_0)\right)\|_2^2 \right] \\
    &\lesssim \dfrac{1}{2n} \sum_{i=1}^n \left( \dfrac{1}{2} \| \Delta_{\Gamma,i} - \Delta_{\Gamma_0,i}\|_F^2 + \|\Delta_{\Gamma,i}^{-1/2}\|_{sp}^2 \|f_i(X_i\beta)-f_i(X_i\beta_0)\|_2^2 \right) \\
    &\lesssim d_n^2(\Gamma,\Gamma_0) + \dfrac{K'^2}{n} \|X(\beta-\beta_0)\|_2^2,
\end{align*}
by Assumption~\ref{hyp_lip} and since $\|\Delta_{\Gamma,i}^{-1}\|_{sp} \lesssim 1$ on $\mathcal{D}_n$. Thus, we obtain finally that $\frac{1}{n} K(p_{0,i},p_{\beta,\Gamma,i})$ and $\frac{1}{n} V(p_{0,i},p_{\beta,\Gamma,i})$ are bounded above by $d_n^2(\Gamma,\Gamma_0)~+~\dfrac{K'^2}{n}~ \|X~(\beta~-~\beta_0)\|_2^2$ up to a multiplicative constant. Then, for $c_1$ large enough,
\begin{align*}
    \Pi(\mathcal{D}_n) &\geq \Pi \left( (\beta,\Gamma)\in \mathcal{B} \times \mathcal{H} \bigg| d_n^2(\Gamma,\Gamma_0)+\dfrac{K'^2}{n} \|X(\beta-\beta_0)\|_2^2 \leq 2 \dfrac{\log(n)}{n}\right) \\
    &\geq \Pi \left(\Gamma\in \mathcal{H} \bigg| d_n^2(\Gamma,\Gamma_0) \leq \dfrac{\log(n)}{n}\right) \Pi \left( \beta\in \mathcal{B} \bigg| \dfrac{K'^2}{n} \|X(\beta-\beta_0)\|_2^2 \leq \dfrac{\log(n)}{n}\right) \\
    &\geq \Pi \left( \Gamma\in \mathcal{H} \bigg| d_n^2(\Gamma,\Gamma_0) \leq \dfrac{\log(n)}{n}\right) \Pi \left( \beta\in \mathcal{B} \bigg| \dfrac{K'^2}{n} \|X\|_{*}^2 \|\beta-\beta_0\|_1^2 \leq \dfrac{\log(n)}{n}\right) \\
\end{align*}
since $\|X\theta\|_2 \leq \|X\|_{*} \|\theta\|_1$.
For the first term, we have that: 
\begin{align*}
    \Pi \left( \Gamma\in \mathcal{H} \bigg| d_n^2(\Gamma,\Gamma_0) \leq \dfrac{\log(n)}{n}\right) &= \Pi \left( \Gamma\in \mathcal{H} \bigg| \dfrac{1}{n} \sum_{i=1}^n \|\Delta_{\Gamma,i}-\Delta_{\Gamma_0,i}\|_F^2 \leq \dfrac{\log(n)}{n}\right) \\
    &\geq \Pi \left( \Gamma\in \mathcal{H} \bigg| \max_i \|\Delta_{\Gamma,i}-\Delta_{\Gamma_0,i}\|_F^2 \leq \dfrac{\log(n)}{n}\right) \\
    &\geq \Pi \left( \Gamma\in \mathcal{H} \bigg| \|\Gamma-\Gamma_0\|_F^2 \leq \dfrac{1}{\overline{\rho_Z}^4}\dfrac{\log(n)}{n}\right) \\
    &=\Pi \left( \Gamma\in \mathcal{H} \bigg| \|\Gamma-\Gamma_0\|_F \leq \dfrac{1}{\overline{\rho_Z}^2}\sqrt{\dfrac{\log(n)}{n}}\right)
\end{align*}
by Lemma \ref{Lem_norm_Gamma} and Assumption \ref{hyp_Zi_sp}. Thus, by Assumption \ref{hyp_gamma0}:
\begin{align*}
    \|\Gamma-\Gamma_0\|_F &= \|\Gamma_0^{1/2} (\Gamma_0^{-1/2}\Gamma\Gamma_0^{-1/2}-Id) \Gamma_0^{1/2}\|_F \\
    &\leq \|\Gamma_0\|_{sp} \|\Gamma_0^{-1/2}\Gamma\Gamma_0^{-1/2}-Id\|_F \\
    &\leq \overline{\rho_{\Gamma_0}} \|\Gamma_0^{-1/2}\Gamma\Gamma_0^{-1/2}-Id\|_F
\end{align*}
By using Lemma \ref{Lem_control_vp}, we obtain that: 
\begin{align*}
    \Pi \bigg( \Gamma\in \mathcal{H} \bigg| d_n^2(\Gamma,\Gamma_0) \leq \dfrac{\log(n)}{n}\bigg) &\geq \Pi \left( \Gamma\in \mathcal{H} \bigg| \|\Gamma_0^{-1/2}\Gamma\Gamma_0^{-1/2}-Id\|_F \leq \overline{\rho_{\Gamma_0}} ^{-1}\overline{\rho_Z}^{-2}\sqrt{\dfrac{\log(n)}{n}}\right) \\
    &\geq \Pi \left( \Gamma\in \mathcal{H} \bigg| \bigcap_{k=1}^{q} \left\{ 1 \leq \rho_k(\Gamma_0^{-1/2}\Gamma\Gamma_0^{-1/2}) \leq 1 + \overline{\rho_{\Gamma_0}} ^{-1}\overline{\rho_Z}^{-2}\sqrt{\dfrac{\log(n)}{\dg n}} \right\}\right)
\end{align*}

Then, denoting by $A= \Gamma_0^{-1/2}\Gamma\Gamma_0^{-1/2} \in \mathbb{R}^{\dg \times \dg }$, since $\Gamma \sim \mathcal{IW}_{\dg}(d, \Sigma)$, we known that $A \sim \mathcal{IW}_{\dg}(d, \Gamma_0^{-1/2}\Sigma\Gamma_0^{-1/2})$, and so $A^{-1} \sim \mathcal{W}_{\dg}(d, \Gamma_0^{1/2}\Sigma^{-1}\Gamma_0^{1/2})$. Then, using Lemma 6.3 of \citet{ning_bayesian_2020}) on the eigenvalues of a Wishart distribution, we obtain that, for large $n$ and since $\dg$ is fixed:

\begin{align*}
    \Pi \bigg( \Gamma\in \mathcal{H} \bigg| d_n^2(\Gamma,\Gamma_0) \leq \dfrac{\log(n)}{n}\bigg) \geq &\left( \dfrac{a_1 t e^2 d}{8 \sqrt{\pi}} \right)^{-\dg} \left( \dfrac{2 d \dg}{e a_1 t} \right)^{-d \dg/2} \left( \dfrac{d}{2e} \right)^{-\dg^2/2} \times \\
    &\det(\Gamma_0^{1/2}\Sigma^{-1}\Gamma_0^{1/2})^{-d/2}
    \exp \left( -\dfrac{a_1 (1+t) \text{Tr}(\Gamma_0^{-1/2}\Sigma\Gamma_0^{-1/2})}{2} \right)
\end{align*}
with $a_1= \left(  1 + \overline{\rho_{\Gamma_0}} ^{-1}\overline{\rho_Z}^{-2}\sqrt{\dfrac{\log(n)}{\dg n}}\right)^{-1}$ and $t=\overline{\rho_{\Gamma_0}} ^{-1}\overline{\rho_Z}^{-2}\sqrt{\dfrac{\log(n)}{\dg n}}$.

Finally, for $n$ large enough, we have that:

$$\log\left(\Pi \left( \Gamma\in \mathcal{H} \bigg| d_n^2(\Gamma,\Gamma_0) \leq \dfrac{\log(n)}{n}\right) \right) \gtrsim -\log(n).$$

Concerning the second term $\Pi \left( \beta\in \mathcal{B} \bigg| \dfrac{K'^2}{n} \|X\|_{*}^2 \|\beta-\beta_0\|_1^2 \leq \dfrac{\log(n)}{n}\right)$ in the lower bound of $\Pi(\mathcal{D}_n)$, by defining $\mathcal{B}_{S_0,n}= \left\{\beta_{S_0} \in \mathbb{R}^{s_0} \bigg| \dfrac{K'}{\sqrt{n}} \|X\|_{*} \|\beta_{S_0}-\beta_{0,S_0}\|_1 \leq \sqrt{\dfrac{\log(n)}{n}} \right\}$ we have that:

\begin{align*}
    \Pi \bigg( \beta\in \mathcal{B} \bigg| \dfrac{K'^2}{n} \|X\|_{*}^2 \|\beta-\beta_0\|_1^2 \leq \dfrac{\log(n)}{n}\bigg) &\geq \Pi \left(S=S_0, \beta\in \mathcal{B} \bigg| \dfrac{K'}{\sqrt{n}} \|X\|_{*} \|\beta-\beta_0\|_1 \leq \sqrt{\dfrac{\log(n)}{n}}\right) \\
    &\geq \dfrac{\pi_p(s_0)}{\binom{\db}{s_0}} \int_{\mathcal{B}_{S_0,n}} g_{S_0}(\beta_{S_0}) d\beta_{S_0} \\
    &\geq \dfrac{\pi_p(s_0)}{\binom{\db}{s_0}} e^{-\lambda \|\beta_0\|_1} \int_{\mathcal{B}_{S_0,n}} g_{S_0}(\beta_{S_0}-\beta_{0,S_0}) d\beta_{S_0} 
\end{align*}
because $g_S$ is the Laplace distribution so satisfy the inequality $g_{S_0}(\beta_{S_0}) \geq e^{-\lambda \|\beta_0\|_1} g_{S_0}(\beta_{S_0}-\beta_{0,S_0})$. Then, since $s_0>0$ by Assumption \ref{hyp_beta0} and using the equation $(6.2)$ of \citet{castillo_bayesian_2015}, we obtain that:

\begin{align*}
    \int_{\mathcal{B}_{S_0,n}} g_{S_0}(\beta_{S_0}-\beta_{0,S_0}) d\beta_{S_0} &\geq e^{-\lambda \frac{\sqrt{\log(n)}}{K' \|X\|_{*}}} \dfrac{\left(\lambda \frac{\sqrt{\log(n)}}{K' \|X\|_{*}}\right)^{s_0}}{s_0!} \\
    &\geq e^{-L_3 \sqrt{\frac{\log(n)}{n}}} \dfrac{\left(\frac{\sqrt{\log(n)}}{L_1 p^{L_2}}\right)^{s_0}}{s_0!}  \\
\end{align*}
by Assumption \ref{hyp_lambda}. We deduce that, by using $\binom{\db}{s_0} s_0! \leq (\db)^{s_0}$: 

\begin{align*}
    \Pi \Bigg( \beta \in \mathcal{B} \bigg| \dfrac{K'^2}{n} \|X\|_{*}^2 \|\beta-\beta_0\|_1^2 \leq \dfrac{\log(n)}{n}\Bigg) &\geq \dfrac{\pi_p(s_0)}{\binom{\db}{s_0}} e^{-\lambda \|\beta_0\|_1} e^{-L_3 \sqrt{\frac{\log(n)}{n}}} \dfrac{\left(\frac{\sqrt{\log(n)}}{L_1 p^{L_2}}\right)^{s_0}}{s_0!} \\
    &\geq \pi_p(s_0) \log(n)^{\frac{s_0}{2}} e^{-(L_2+1)s_0\log(p)-s_0\log(q)-L_3\sqrt{\frac{\log(n)}{n}}-s_0\log(L_1) - \lambda \|\beta_0\|_1}\\
    &\gtrsim \pi_p(s_0) \exp\left(- \Tilde{C} s_0\log(p) \right)
\end{align*}
for a constant $\Tilde{C}$ since $s_0 + \sqrt{\frac{\log(n)}{n}}+s_0\log(q) + s_0\log(p) \lesssim s_0\log(p)$ and $\lambda \|\beta_0\|_1 \lesssim s_0\log(p)$ by Assumption \ref{hyp_beta0}.
Finally, there exists a constant $L$ such that:

\begin{align*}
    \Pi(\mathcal{D}_n) &\geq \Pi \left( \Gamma\in \mathcal{H} \bigg| d_n^2(\Gamma,\Gamma_0) \leq \dfrac{\log(n)}{n}\right) \Pi \left( \beta\in \mathcal{B} \bigg| \dfrac{K'^2}{n} \|X\|_{*}^2 \|\beta-\beta_0\|_1^2 \leq \dfrac{\log(n)}{n}\right) \\
    &\geq \pi_p(s_0) e^{-L (s_0\log(p) + \log(n) )}
\end{align*}

Thus, we have shown that $e^{-(1+C)c_1 \log(n)}\Pi(\mathcal{D}_n) \geq \pi_p(s_0) e^{-M(s_0 \log(p) + \log(n))}$ for some constant $M$, as required to conclude the proof of Lemma \ref{Lem_denom}.

\subsection{Proof of Lemma \ref{Lem_test}}
\label{Proof_Lem2}

Let $(\beta_1, \Gamma_1)\in \mathcal{B} \times \mathcal{H}$ such that $R_n(p_0,p_1) \geq \epsilon_n^2$. 
First, for testing $H_0$: $p=p_0$ against $H_1$: $p=p_1$, consider the most powerful test $\overline{\varphi}_n=\mathds{1}_{\Lambda_n(\beta_1,\Gamma_1)\geq 1}$ given by the Neyman-Pearson lemma, where $\Lambda_n(\beta_1,\Gamma_1)=\dfrac{p_1}{p_0}$ be the likelihood ratio of $p_1$ and $p_{0}$.
Thus, 
\begin{align*}
    \mathbb{E}_0[\overline{\varphi}_n] &= \mathbb{P}_0(\sqrt{\Lambda_n(\beta_1,\Gamma_1)}\geq 1)= \int \mathds{1}_{\sqrt{p_1(y)}\geq \sqrt{p_0(y)}} p_0(y)dy \\
    &\leq \int \sqrt{p_0(y) p_1(y)} dy = e^{-n R_n(p_0,p_1)} \leq e^{-n\epsilon_n^2}
\end{align*}
by assumption on $(\beta_1, \Gamma_1)$. This proves the first result of the lemma.

Then, for the second part of the lemma, note that:

\begin{equation}
\label{Eq_esperance_p1}
    \mathbb{E}_1[1-\overline{\varphi}_n] = \mathbb{P}_1(\sqrt{\Lambda_n(\beta_1,\Gamma_1)}\leq 1) \leq \int \sqrt{p_0(y) p_1(y)} dy \leq e^{-n\epsilon_n^2}
\end{equation}
However, by using Cauchy-Schwarz inequality:
\begin{align*}
    \mathbb{E}_{\beta, \Gamma}[1-\overline{\varphi}_n] &= \int (1-\overline{\varphi}_n(y))\dfrac{p_{\beta, \Gamma}(y)}{p_1(y)}dp_1(y) \\
    &\leq \mathbb{E}_{1}[1-\overline{\varphi}_n]^{1/2} \mathbb{E}_{1}\left[\left(\dfrac{p_{\beta, \Gamma}}{p_1}\right)^2\right]^{1/2} \\
    &\leq e^{-n\epsilon_n^2/2} \mathbb{E}_{1}\left[\left(\dfrac{p_{\beta, \Gamma}}{p_1}\right)^2\right]^{1/2}
\end{align*}
by Equation \eqref{Eq_esperance_p1}. Therefore, the test $\overline{\varphi}_n$ can also have exponentially small error of type II at other alternatives if we can controlled the second term: we want to show that $\mathbb{E}_{1}\left[\left(\dfrac{p_{\beta, \Gamma}}{p_1}\right)^2\right]^{1/2}\leq e^{7n\epsilon_n^2/16}$ for every $(\beta, \Gamma) \in \mathcal{F}_{1,n}$.

Recall that here $p_{\beta, \Gamma}= \prod_{i=1}^n \mathcal{N}_{m_i}(f_i(X_i \beta), \Delta_{\Gamma,i})$, where $\Delta_{\Gamma,i}=Z_i \Gamma Z_i^\top + \sigma^2 I_{m_i}$. By denoting $\Delta_{\Gamma,i}^{*} ~ = ~\Delta_{\Gamma,i}^{-1/2} \Delta_{\Gamma_1,i} \Delta_{\Gamma,i}^{-1/2}$, then for $(\beta, \Gamma) \in \mathcal{F}_{1,n}$, if $2 \Delta_{\Gamma,i}^{*} -Id$ and $2Id -\Delta_{\Gamma,i}^{*^{-1}}$ are non-singular matrices for every $i \in \{1,\cdots,n\}$, we can show that: 

\begin{align}
    \mathbb{E}_{1}\left[\left(\dfrac{p_{\beta, \Gamma}}{p_1}\right)^2\right] = &\prod_{i=1}^n \left[ \det(\Delta_{\Gamma,i}^{*})^{1/2} \det(2Id-\Delta_{\Gamma,i}^{*^{-1}})^{-1/2} \right] \times \notag\\
    & \exp \left\{\sum_{i=1}^n \| (2\Delta_{\Gamma,i}^{*}-Id)^{-1/2} \Delta_{\Gamma,i}^{-1/2} (f_i(X_i\beta)-f_i(X_i\beta_1))\|_2^2 \right\} \label{eq_esp}.
\end{align}

Let us now prove that these matrices are non-singular. We have, for all $k\leq m_i$,

\begin{equation*}
    \max_{1\leq i \leq n} \|\Delta_{\Gamma,i}^{*}-Id\|_{sp} =  \max_{1\leq i \leq n} \rho_{max}(|\Delta_{\Gamma,i}^{*}-Id|) \geq \max_{1\leq i \leq n} | \rho_k(\Delta_{\Gamma,i}^{*})-1|.
\end{equation*}
Note that
\begin{align*}
    \max_{1\leq i \leq n} \|\Delta_{\Gamma,i}^{*}-Id\|_{sp} &= \max_{1\leq i \leq n} \|\Delta_{\Gamma,i}^{-1/2}(\Delta_{\Gamma_1,i}-\Delta_{\Gamma,i})\Delta_{\Gamma,i}^{-1/2}\|_{sp} \\
    &\leq  \max_{1\leq i \leq n} \|\Delta_{\Gamma,i}^{-1}\|_{sp} \|\Delta_{\Gamma_1,i}-\Delta_{\Gamma,i}\|_{F} \\
    &\leq \max_{1\leq i \leq n} \|\Delta_{\Gamma,i}^{-1}\|_{sp} d_n(\Gamma,\Gamma_1) \\
    &\leq \dfrac{\epsilon_n^2}{2M_{\text{obs}}}
\end{align*}
by Lemma \ref{Lem_norm_Gamma} and since $(\beta,\Gamma) \in \mathcal{F}_{1,n}$. Thus, for all $k\leq m_i$, $\max_{1\leq i \leq n} | \rho_k(\Delta_{\Gamma,i}^{*})-1| \leq \dfrac{\epsilon_n^2}{2M_{\text{obs}}}$. We deduce that

\begin{equation}
    \label{eq_borne_vp}
    1-\dfrac{\epsilon_n^2}{2M_{\text{obs}}} \leq \min_{1\leq i \leq n} \rho_{\min}(\Delta_{\Gamma,i}^{*}) \leq \max_{1\leq i \leq n} \rho_{\max}(\Delta_{\Gamma,i}^{*})\leq 1+\dfrac{\epsilon_n^2}{2M_{\text{obs}}}.
\end{equation}

Therefore, since $\dfrac{\epsilon_n^2}{2M_{\text{obs}}} \underset{n \rightarrow \infty}{\longrightarrow} 0$ by Assumption \ref{hyp_beta0}, and for all $k \leq m_i$, $\rho_k(2 \Delta_{\Gamma,i}^{*} -Id)=2\rho_k( \Delta_{\Gamma,i}^{*})-1$ and $\rho_k(2Id -\Delta_{\Gamma,i}^{*^{-1}})=2 - \rho_k( \Delta_{\Gamma,i}^{*^{-1}})=2-\rho_k^{-1}( \Delta_{\Gamma,i}^{*})$, we deduce that $2 \Delta_{\Gamma,i}^{*} -Id$ and $2Id -\Delta_{\Gamma,i}^{*^{-1}}$ are non-singular on $\mathcal{F}_{1,n}$ for every $i \in \{1,\cdots,n\}$. 

For concluding the proof, it remains to bound the right side term of \eqref{eq_esp}.
By using \eqref{eq_borne_vp} and the inequalities $(1-x^2)/(1-2x) \leq 1+3x$ for $x>0$ small, and $1+x\leq e^x$, we obtain for $n$ large enough:
\begin{align*}
    \det(\Delta_{\Gamma,i}^{*})^{1/2} \det(2Id-\Delta_{\Gamma,i}^{*^{-1}})^{-1/2} &= \left(\prod_{k=1}^{m_i} \rho_k(\Delta_{\Gamma,i}^{*})\right)^{1/2} \left(\prod_{k=1}^{m_i} 2-\rho_k^{-1}(\Delta_{\Gamma,i}^{*})\right)^{-1/2} \\
    &= \left(\prod_{k=1}^{m_i} \dfrac{\rho_k(\Delta_{\Gamma,i}^{*})}{2-\rho_k^{-1}(\Delta_{\Gamma,i}^{*})}\right)^{1/2} \\
    &\leq \left( \dfrac{1-\dfrac{\epsilon_n^4}{4M_{\text{obs}}^2}}{1+\dfrac{\epsilon_n^2}{M_{\text{obs}}}} \right)^{m_i/2} \leq \left(1+3\dfrac{\epsilon_n^2}{2M_{\text{obs}}} \right)^{m_i/2} \\
    &\leq \exp\left(3\dfrac{m_i\epsilon_n^2}{4M_{\text{obs}}} \right) \leq e^{3\epsilon_n^2/4}
\end{align*}
Moreover, for $n$ large enough,
\begin{align*}
    \sum_{i=1}^n &\| (2\Delta_{\Gamma,i}^{*}-Id)^{-1/2} \Delta_{\Gamma,i}^{-1/2} (f_i(X_i\beta)-f_i(X_i\beta_1))\|_2^2 \\
    &\leq \max_{1\leq i \leq n} \|(2\Delta_{\Gamma,i}^{*}-Id)^{-1}\|_{sp} \max_{1\leq i \leq n} \|\Delta_{\Gamma,i}^{-1}\|_{sp} \sum_{i=1}^n \|f_i(X_i\beta)-f_i(X_i\beta_1)\|_2^2 \\
    &\leq 2 \gamma_n \dfrac{n \epsilon_n^2}{16 \gamma_n} = \dfrac{n\epsilon_n^2}{8} 
\end{align*}
since $(\beta,\Gamma) \in \mathcal{F}_{1,n}$. Finally, by using \eqref{eq_esp}, we conclude that $\mathbb{E}_{1}\left[\left(\dfrac{p_{\beta, \Gamma}}{p_1}\right)^2\right]^{1/2} \leq e^{3n\epsilon_n^2/8} e^{n\epsilon_n^2/16}=e^{7n\epsilon_n^2/16}$, and so $\sup_{(\beta, \Gamma) \in \mathcal{F}_{1,n}} \mathbb{E}_{\beta, \Gamma}[1-\overline{\varphi}_n] \leq e^{-n \epsilon_n^2/16}$, which concludes the proof.

\section{Useful lemmas}\label{appB}
\begin{lemma}
    \label{Lem_rho_max}
    For $A$ and $B$ two matrices, $$\rho_{min}(B) \|A\|_{sp}^2 \leq \rho_{max}(ABA^\top) \leq \rho_{max}(B) \|A\|_{sp}^2.$$
\end{lemma}

\begin{proof}
By the Courant–Fischer–Weyl min-max principle,
    \begin{align*}
    \rho_{max}(A B A^\top) &= \underset{x\neq0}{\max} \dfrac{\langle A B A^\top x, x \rangle}{\|x\|^2} \\
    &= \underset{x\neq0}{\max} \dfrac{\langle B A^\top x, A^\top x \rangle}{\|x\|^2} \\
    &= \underset{x\neq0}{\max} \dfrac{\langle B A^\top x, A^\top x \rangle}{\|A^\top x\|^2} \dfrac{\|A^\top x\|^2}{\|x\|^2} \\
    &\leq \rho_{max}(B) \underset{x\neq0}{\max} \dfrac{\|A^\top x\|^2}{\|x\|^2} \\
    &= \rho_{max}(B) \rho_{max}(A A^\top) \\
    &= \rho_{max}(B) \| A \| _{sp}^2 
\end{align*}
We obtain the other inequality with similar arguments.
\end{proof}

\begin{lemma}
    \label{vp_Delta}
    Grant Assumptions~\ref{hyp_gamma0} and \ref{hyp_Zi_sp}. Thus, $\Delta_{\Gamma_0,i}:= Z_i \Gamma_0 Z_i^\top + \sigma^2 I_{m_i}$ satisfies:
\[1 \lesssim \min_i \rho_{min}(\Delta_{\Gamma_0,i}) \leq \max_i \rho_{max}(\Delta_{\Gamma_0,i}) \lesssim 1\]
\end{lemma}

\begin{proof}
By the Weyl's inequality, for $1 \leq i \leq n$,
\[\rho_{min}(\Delta_{\Gamma_0,i}) \geq \rho_{min}(Z_i \Gamma_0 Z_i^\top) + \sigma^2 \geq \sigma^2 \]
since $Z_i \Gamma_0 Z_i^\top$ is a positive definite matrix. Thus $\min_i \rho_{min}(\Delta_{\Gamma_0,i}) \geq \sigma^2$, otherwise, $\min_i \rho_{min}(\Delta_{\Gamma_0,i}) \gtrsim 1$.

For the other inequality, by the Weyl's inequality, we have that
\[\rho_{max}(\Delta_{\Gamma_0,i}) \leq \rho_{max}(Z_i \Gamma_0 Z_i^\top) + \sigma^2.\] 
Then, by Lemma~\ref{Lem_rho_max}, we have that:
$$\rho_{max}(Z_i \Gamma_0 Z_i^\top) \leq \rho_{max}(\Gamma_0) \| Z_i \| _{sp}^2 $$
and by Assumptions~\ref{hyp_gamma0} and \ref{hyp_Zi_sp},
\[\max_i \rho_{max}(\Delta_{\Gamma_0,i}) \leq \rho_{max}(\Gamma_0) \max_i \| Z_i \| _{sp}^2 + \sigma^2 \lesssim 1. \]
\end{proof}

\begin{lemma}
    \label{Lem_norm_Gamma}
    For $\Gamma_1, \Gamma_2 \in \mathcal{H}$, under Assumptions~\ref{hyp_ni_sup_q}, \ref{hyp_rg_Zi} and \ref{hyp_Zi_sp}, we have that
$$\max_i \|\Delta_{\Gamma_1,i}-\Delta_{\Gamma_2,i} \|_F^2 \lesssim  \|\Gamma_1-\Gamma_2 \|_F^2 \lesssim d_n^2(\Gamma_1,\Gamma_2)=\frac{1}{n}\sum_{i=1}^n \|\Delta_{\Gamma_1,i}-\Delta_{\Gamma_2,i}\|_F^2.$$
\end{lemma}

\begin{proof}
First,
    \[\| \Delta_{\Gamma_1,i}-\Delta_{\Gamma_2,i} \|_F^2 = \| Z_i (\Gamma_1 - \Gamma_2) Z_i^\top\|_F^2 \leq \| Z_i \| _{sp}^4 \|\Gamma_1 - \Gamma_2\|_F^2,\]
    since $\|AB\|_F \leq \|A\|_{sp} \|B\|_F$, and by Assumption~\ref{hyp_Zi_sp}, we have that $$ \max_{1 \leq i \leq n}\| \Delta_{\Gamma_1,i}-\Delta_{\Gamma_2,i} \|_F^2 \lesssim \|\Gamma_1 - \Gamma_2\|_F^2.$$ 
    Then, by Assumption~\ref{hyp_rg_Zi}, for each $i$ such that $m_i \geq \dg$,
    $Z_i^\top Z_i$ is invertible and
    \begin{align*}
        \|\Gamma_1 - \Gamma_2\|_F^2 &= \|(Z_i^\top Z_i)^{-1}Z_i^\top Z_i(\Gamma_1 - \Gamma_2) Z_i^\top Z_i (Z_i^\top Z_i)^{-1}\|_F^2 \\
        &\leq \| Z_i (\Gamma_1 - \Gamma_2) Z_i^\top\|_F^2 \| (Z_i^\top Z_i)^{-1} Z_i^\top \| _{sp}^4.
    \end{align*}
    Then, by Assumption~\ref{hyp_ni_sup_q}, we have that
    \begin{align*}
        \max_{1 \leq i \leq n}\| \Delta_{\Gamma_1,i}-\Delta_{\Gamma_2,i} \|_F^2 &\lesssim \|\Gamma_1 - \Gamma_2\|_F^2 \\
        &\lesssim \dfrac{1}{\sum_{i=1}^n \mathds{1}_{m_i \geq \dg}} \sum_{i : m_i \geq \dg} \| Z_i (\Gamma_1 - \Gamma_2) Z_i^\top\|_F^2 \| (Z_i^\top Z_i)^{-1} Z_i^\top \| _{sp}^4 \\
    &\lesssim \dfrac{1}{n} \sum_{i : m_i \geq \dg} \| Z_i (\Gamma_1 - \Gamma_2) Z_i^\top\|_F^2 \| (Z_i^\top Z_i)^{-1} Z_i^\top \| _{sp}^4 \\
    &\lesssim \dfrac{1}{n} \sum_{i=1}^n \| Z_i (\Gamma_1 - \Gamma_2) Z_i^\top\|_F^2  = d_n^2(\Gamma_1,\Gamma_2).
    \end{align*}
where the last inequality uses Assumptions \ref{hyp_rg_Zi} and \ref{hyp_Zi_sp}.
\end{proof}

\begin{lemma}
\label{Lem_control_vp}
    For a positive definite symmetric matrix $A \in \mathbb{R}^{\dg\times \dg}$ such as its eigenvalues satisfy \\$1\leq \rho_1(A) \leq ... \leq\rho_q(A) \leq1 + \dfrac{\epsilon}{\sqrt{\dg}}$, then $\|A-I_{\dg}\|_F \leq \epsilon$. 
\end{lemma}

\begin{proof}
Observe that
\begin{align*}
    \|A-I_{\dg}\|_F \leq \epsilon &\Leftrightarrow \text{Tr}((A-I_{\dg})^2) \leq \epsilon^2 \quad \text{since } A \text{ is symmetric} \\
    &\Leftrightarrow \sum_{k=1}^\dg \rho_k(A-I_\dg)^2 \leq \epsilon^2 \\
    &\Leftrightarrow \sum_{k=1}^\dg (\rho_k(A)-1)^2 \leq \epsilon^2 .
\end{align*}
By assumption, for $1 \leq k \leq \dg$, we have that $0 \leq \rho_k(A) -1 \leq\dfrac{\epsilon}{\sqrt{\dg}}$ and then $\max_{1 \leq k \leq \dg} (\rho_k(A) -1)^2 \leq \dfrac{\epsilon^2}{\dg}$. 
Hence, since $\sum_{k=1}^\dg (\rho_k(A)-1)^2 \leq \dg \times \max_{1 \leq k \leq \dg} (\rho_k(A) -1)^2 \leq \epsilon^2$ and so $\|A-I_\dg\|_F \leq \epsilon$.
\end{proof}

\end{appendices}

\end{document}